\documentclass[11pt]{article}
\usepackage{multicol}
\usepackage{amssymb,amsmath,amsthm,bm}
\usepackage[colorlinks=true, urlcolor=blue,linkcolor=blue, citecolor=blue]{hyperref}
\usepackage{mathrsfs,verbatim}
\usepackage{caption,cleveref}
\usepackage{graphicx, float}
\usepackage{color, mathtools, tikz, xcolor}
\usepackage{enumerate,enumitem}
\usepackage{tikz}
\usetikzlibrary{shapes.geometric}
\usetikzlibrary{decorations.markings}
\usepackage{bbm}
\usepackage[noadjust]{cite}

\usetikzlibrary{calc}

\usepackage[mathscr]{euscript}
\usepackage{geometry}
\numberwithin{equation}{section}

\begin{document}

	\newtheorem{theorem}{Theorem}[section]
	\newtheorem{problem}[theorem]{Problem}
	\newtheorem{exercise}[theorem]{Exercise}
	\newtheorem{corollary}[theorem]{Corollary}
	\newtheorem{conjecture}[theorem]{Conjecture}
	\newtheorem{claim}[theorem]{Claim}
	\newtheorem{proposition}[theorem]{Proposition}
	\newtheorem{lemma}[theorem]{Lemma}
	\newtheorem{definition}[theorem]{Definition}
	\newtheorem{example}[theorem]{Example}
	\newtheorem{remark}[theorem]{Remark}
	\newtheorem{solution}[theorem]{Solution}
	\newtheorem{case}{Case}[theorem]
	\newtheorem{condition}[theorem]{Condition}
	\newtheorem{assumption}[theorem]{Assumption}
	\newtheorem{note}[theorem]{Note}
	\newtheorem{notes}[theorem]{Notes}	\newtheorem{observation}[theorem]{Observation}

	\newenvironment{proofclaim}[1][Proof of claim]{\begin{proof}[#1]\renewcommand*{\qedsymbol}{\(\blacksquare\)}}{\end{proof}}
	
	\newcommand{\indicate}{\mathbbm{1}}
	
	\setlist{nolistsep}
	
	\def\al#1{}
	\renewcommand{\al}[1]{\footnote{\textcolor{blue}{\textbf{AL-Jun2023 :}#1}}} 
	\def\COMMENT#1{}
	\renewcommand{\COMMENT}[1]{\footnote{ \textcolor{red}{ #1}}}

	\title{Cycle Partitions in Dense Regular Digraphs and Oriented Graphs}
	\author{Allan Lo\footnote{School of Mathematics, University of Birmingham, United Kingdom, Email: s.a.lo@bham.ac.uk. A.~Lo was partially supported by EPSRC, grant no. EP/V002279/1 and EP/V048287/1. There are no additional data beyond that contained within the main manuscript.} \hspace{0.2in} Viresh Patel\footnote{School of Mathematical Sciences, Queen Mary University of London, United Kingdom,  Email: viresh.patel@qmul.ac.uk. V.~Patel  was supported by the Netherlands Organisation for Scientific Research (NWO)
			through the Gravitation Programme Networks (024.002.003).} \hspace{0.2in} Mehmet Akif Yıldız\footnote{Korteweg de Vries Instituut voor Wiskunde, Universiteit van Amsterdam, The Netherlands. Email: m.a.yildiz@uva.nl.  M.A.~Yıldız was supported by a Marie Skłodowska-Curie Action from the EC (COFUND grant no. 945045 ) and by the NWO Gravitation project NETWORKS  (grant no. 024.002.003). }   } \vspace{0.2in}
	
	\maketitle
	\begin{abstract}
		A conjecture of Jackson from 1981 states that every $d$-regular oriented graph on $n$ vertices with $n\leq 4d+1$ is Hamiltonian. We prove this conjecture for sufficiently large $n$. In fact we prove a more general result that for all $\alpha>0$, there exists $n_0=n_0(\alpha)$ such that every $d$-regular digraph on $n\geq n_0$ vertices with $d \ge \alpha n $ can be covered by at most $n/(d+1)$ vertex-disjoint cycles, and moreover that if $G$ is an oriented graph, then at most $n/(2d+1)$ cycles suffice.
		\bigskip
		
		
	\end{abstract}

	\section{Introduction}

	A \textit{Hamilton cycle} in a (directed) graph is a (directed) cycle that visits every vertex. Hamilton cycles are one of the most intensely studied structures in graph theory and there are numerous results that establish (best-possible) conditions guaranteeing their existence.
	The seminal result of Dirac~\cite{Dirac} states that every graph on $n\ge 3$ vertices with minimum degree at least $n/2$ is Hamiltonian. Ghouila-Houri~\cite{GhouilaHouri} showed the corresponding version in directed graphs (\textit{digraph} for short), that is, every digraph on $n\ge 3$ vertices with minimum semi-degree at least $n/2$ (i.e.\ every vertex has in- and outdegree at least $n/2$) is Hamiltonian. These bounds are tight by taking e.g.\ the disjoint union of two cliques (a regular extremal example) or a slightly imbalanced complete bipartite (di)graph (an irregular extremal example). 
Recall that an \textit{oriented graph} is a digraph that can have at most one edge between each pair of vertices (whereas a digraph can have up to two, one in each direction).
	For oriented graphs, a more recent result of Keevash, K{\"u}hn, and Osthus~\cite{OrientedExactMinDegree} established a (tight) minimum semi-degree threshold of $\lceil (3n - 4)/8 \rceil$ for Hamiltonicity. In contrast to graphs and digraphs, there are no \textit{regular} extremal examples for this result.
	Jackson~\cite{JacksonConjecture} conjectured in 1981 that regularity actually reduces the degree threshold significantly for oriented graphs: 
	
	\begin{conjecture}[Jackson~\cite{JacksonConjecture}] \label{conj:Jackson}
		For each $d > 2$, every $d$-regular oriented graph (i.e. every vertex has $d$ in- and outneighbours) on $n\leq 4d+1$ vertices has a Hamilton cycle.
	\end{conjecture}
	
	The disjoint union of two regular tournaments shows that Jackson's conjecture is best possible. This example works for $n\equiv 2\pmod{4}$ (since regular tournaments require an odd number of vertices), but similar examples exist for $n\equiv 0,1,3\pmod{4}$ (see Section~\ref{sec:examples+proofs}). 
	
	We note that an approximate version of Jackson's conjecture was recently verified by the current authors in~\cite{LPY}, that is, for every $\varepsilon>0$, there exists $n_0(\varepsilon)$ such that every $d$-regular oriented graph on $n\ge n_0(\varepsilon)$ vertices with $d\ge (1/4+\varepsilon)n$ is Hamiltonian. Here, we verify the exact version for large $n$.
	
	\begin{theorem}\label{thm:Jackson}
		There exists an integer $n_0$ such that every $d$-regular oriented graph on $n\ge n_0$ vertices with $n\leq 4d+1$ has a Hamilton cycle. 
	\end{theorem}

	Generalizing questions about Hamilton cycles, one can consider the question of covering the vertices of a (di)graph with as few vertex-disjoint cycles as possible. Indeed, we prove Theorem~\ref{thm:Jackson} by showing a more general result about covering regular digraphs and oriented graphs with few vertex-disjoint cycles.

	\begin{theorem} \label{thm:main}
		For all $\alpha >0$, there exists $n_0 = n_0(\alpha)$ such that every $d$-regular digraph
		$G$ on $n\geq n_0$ vertices with $d \ge \alpha n $ can be covered by at most $n/(d+1)$ vertex-disjoint cycles.
		Moreover if $G$ is an oriented graph, then at most $n/(2d+1)$ cycles suffice.
	\end{theorem}	
	
	This is best possible by considering the disjoint union of complete digraphs of order $d+1$ for digraphs and the disjoint union of regular tournaments of order $2d+1$ for oriented graphs. Notice that we have $n/(2d+1)<2$ when $n\leq 4d+1$, so that Theorem~\ref{thm:main} implies Theorem~\ref{thm:Jackson}. We also point out that the proof of Theorem~\ref{thm:main} in fact shows that each cycle is relatively long (of length at least $d/2$).
 
 Theorem~\ref{thm:main} generalizes the following result of Gruslys and Letzter~\cite{GruslysLetzter} from regular graphs to regular digraphs and oriented graphs.

	\begin{theorem}[Gruslys and Letzter~\cite{GruslysLetzter}] \label{thm:GruslysLetzter}
		For all $\alpha >0$, there exists $n_0 = n_0(\alpha)$ such that every $d$-regular graph on $n \ge n_0$ vertices with $d \ge \alpha n $ can be covered by at most $n/(d+1)$ vertex-disjoint cycles.
	\end{theorem}
	Theorem~\ref{thm:main} implies Theorem~\ref{thm:GruslysLetzter} by making every edge into a directed $2$-cycle.
	
	\subsection{Related work}
	Theorem~\ref{thm:main} also has connections with several well-studied problems in extremal graph theory: here we mention some of them.

	\subsubsection{Path cover}
	
	A weaker version of cycle cover is path cover. The \textit{path cover number} of a (di)graph $G$, denoted by $\pi(G)$, is the minimum number of vertex-disjoint (directed) paths needed to cover $V(G)$. This was introduced by Ore~\cite{ORE-PathCover}, and he showed that $\pi(G)\leq n-\sigma_2(G)$ holds where $\sigma_2(G)$ denotes the minimum sum of the degrees over all non-adjacent vertices. Magnant and Martin~\cite{PathCover} conjectured that \textit{regularity} significantly reduces the upper bound for $\pi(G)$:
	
	\begin{conjecture}[Magnant and Martin~\cite{PathCover}]
		\label{conj:PATH-COVER}
		If $G$ is a $d$-regular graph on $n$ vertices, then $\pi(G)\leq n/(d+1)$.
	\end{conjecture}
	
	It is known that Conjecture~\ref{conj:PATH-COVER} holds for small values of $d$ (see~\cite{PathCover} for $d\leq 5$ and see~\cite{Feige-Fuchs} for $d=6$).  
	Han~\cite{Han-Path-Cover} showed that, for dense graphs, it is enough to use $1+n/(d+1)$ paths to cover almost all vertices. Also, Theorem~\ref{thm:GruslysLetzter} verifies Conjecture~\ref{conj:PATH-COVER} in the dense case. It is worth noting that the Linear Arboricity Conjecture~\cite{linear-arboricity} implies Conjecture~\ref{conj:PATH-COVER} for odd values of $d$, and gives $\pi(G)\leq 2n/(d+2)$ for general $d$ (see~\cite{Feige-Fuchs} for a detailed discussion).
	
	For digraphs, the classical result of Gallai and Milgram~\cite{GAL-MIL} states that $\pi(G)$ can be bounded above by the size of the maximum independent set (and Dilworth's~\cite{DILWORTH} theorem says that equality holds for the special case of posets).
	As our Theorem~\ref{thm:main} generalizes Theorem~\ref{thm:GruslysLetzter} from graphs to digraphs and oriented graphs, we believe the following stronger version of Conjecture~\ref{conj:PATH-COVER} holds, which Theorem~\ref{thm:main} establishes in the dense case. 
	
	\begin{conjecture}
		\label{conj:DIRECTED-PATH-COVER}
		If $G$ is a $d$-regular digraph on $n$ vertices, then $\pi(G)\leq n/(d+1)$. Moreover, $\pi(G)\leq n/(2d+1)$ holds if $G$ is oriented. 
	\end{conjecture}
	
	Also, Conjecture~\ref{conj:DIRECTED-PATH-COVER} implies Conjecture~\ref{conj:PATH-COVER} by making every edge into a directed $2$-cycle.
	
	\subsubsection{Extending perfect matchings}
	
	Gruslys and Letzter~\cite{GruslysLetzter}, as well as proving Theorem~\ref{thm:GruslysLetzter}, proved that every large $d$-regular \emph{bipartite graph}
	$G$ on~$n$ vertices with $d$ linear in $n$ can be covered by at most $n/2d$ vertex-disjoint paths.
	They mentioned that one should be able to replace paths by cycles. Indeed, as a corollary of Theorem~\ref{thm:main}, the result below shows that those cycles can be found in such a way that they even contain any prescribed perfect matching.
	
	\begin{corollary} \label{thm:reg-bibpartite}
		For all $\alpha >0$, there exists $n_1 = n_1(\alpha)$ such that, for every $d$-regular bipartite graph on $n \ge n_1$ vertices with $d \ge \alpha n $, any perfect matching can be extended to vertex-disjoint cycles covering all vertices with at most $n/2d$ cycles.
	\end{corollary}
	
	\begin{proof}[Proof of Corollary~$\ref{thm:reg-bibpartite}$]
		Let $\alpha > 0$ and $n_1 = 2 n_0(\alpha)$, where $n_0$ is the function given in Theorem~\ref{thm:main}.
		Let $G$ be a $d$-regular bipartite graph on $n \ge n_1$ vertices with $d \ge \alpha n $ and vertex classes $X$ and~$Y$. 
		Since $G$ is bipartite and regular, $n$ is even and $|X| = |Y| = n/2$.
		Let $M$ be any perfect matching of~$G$.
		Let $X = \{ x_1, \dots, x_{n/2}\}$ and $Y = \{y_1, \dots, y_{n/2}\}$ be such that $x_iy_i \in E(M)$ for all~$i$. 
		Define the digraph~$H$ on~$X$ such that for any distinct $i,j\in [n/2]$, $x_ix_j \in E(H)$ if and only if $x_i  y_j \in E(G)$. 
		Note that $H$ is $(d-1)$-regular on $n/2$ vertices. 
		By Theorem~\ref{thm:main}, $H$ can be covered by at most $n/2d$ vertex-disjoint cycles.
		Note that a (directed) cycle $x_{i_1} x_{i_2}\dots x_{i_\ell}$ in~$H$ corresponds to a cycle $x_{i_1} y_{i_2 } x_{i_2} y_{i_3}\dots x_{i_\ell} y_{i_1}$ in~$G$.
		Therefore $G$ can be covered by at most $n/2d$ vertex-disjoint cycles that contain all the edges of~$M$.
	\end{proof}
	
	Note that Corollary~\ref{thm:reg-bibpartite} is tight by considering the disjoint union of $n/2d$ many $K_{d,d}$'s. It also shows that $d$-regular bipartite graphs on $n$ (sufficiently large) vertices with $d>n/4$ are examples of graphs in which every perfect matching can be extended into a Hamilton cycle. This property is called the PMH-property in~\cite{Line-Graphs-PMH}. H{\"a}ggkvist~\cite{Haggvist-F-Hamiltonian} initiated the study of sufficient conditions for the PMH-property (using the name $F$-Hamiltonian where $F$ is a perfect matching) by showing $\sigma_2(G)\geq n+1$ is sufficient. Las Vergnas~\cite{las-vergnas} proved a similar condition
	for bipartite graphs, which also (almost) implies Corollary~\ref{thm:reg-bibpartite} in the case $d>n/4$, and Yang~\cite{Yang-F-Hamiltonian} gave minimum edge density conditions to guarantee the PMH property in graphs and bipartite graphs. In the sparse setting, as a special case of a conjecture of Ruskey and Savage~\cite{HYPERCUBE-MATCHING}, Fink~\cite{FINK-HYPERCUBE} proved that the hypercube has the PMH-property.

	\subsection{Extremal examples for Conjecture~\ref{conj:Jackson}}
	\label{sec:examples+proofs}
	
	We end the introduction with the deferred examples from earlier. Note that we give a general overview of the paper at the end of the next section.
	
	Recall that the disjoint union of two regular tournaments on $n/2$ vertices shows that Jackson's conjecture (Conjecture~\ref{conj:Jackson}) is best possible. This example works when $n \equiv 2 \pmod{4}$ and similar examples for other values of $n$ are constructed as follows.
	When $n \equiv 0 \pmod{4}$, we take the disjoint union of two complete graphs on $n/2$ vertices, remove a perfect matching, and orient the edges so that the resulting oriented graph is regular. When $n \equiv 3 \pmod{4}$, we take a disjoint union of a regular tournament on $(n-1)/2$ vertices and a regular orientation of the complete graph minus a perfect matching on $(n+1)/2$ vertices. 
	When $n \equiv 1 \pmod{4}$, we take a regular orientation of a disjoint union of a clique minus a perfect matching on $(n-1)/2$ vertices and a clique minus a Hamilton cycle on $(n+1)/2$ vertices.

	\section{Notation and preliminaries}

	Throughout the paper, we use standard graph theory notation and terminology.
	For $k\in\mathbb{N}$, we sometimes denote the set $\{1,2,\ldots,k\}$ by $[k]$. 
	For a digraph~$G$, we denote its vertex set by $V(G)$ and its edge set $E(G)$, and sometimes write
	$|G| := |V(G)|$ and $e(G) := |E(G)|$. For $a,b \in V(G)$, we write $ab$ for the directed edge from $a$ to~$b$. We write $H \subseteq G$ to mean $H$ is a \textit{subdigraph} of $G$, i.e.\ $V(H) \subseteq V(G)$ and $E(H) \subseteq E(G)$.
	We sometimes think of $F \subseteq E(G)$ as a subdigraph of $G$ with vertex set consisting of those vertices incident to edges in~$F$ and with edge set~$F$. We write $G-F$ for the digraph obtained from $G$ by deleting the edges in~$F$.
	For $S \subseteq V(G)$, we write $G[S]$ for the subdigraph of $G$ induced by $S$ and $G-S$ for the digraph $G[V(G) \setminus S]$. 
	For $A,B \subseteq V(G)$ not necessarily disjoint, we define $E_G(A,B):= \{ab \in E(G): a \in A, \; b \in B \}$ and we write  $e_G(A,B) := |E_G(A,B)|$\footnote{ We DO NOT write $G[A,B]$ for the graph with vertex set $A \cup B$ and edge set $E_G(A,B)$, but instead for a bipartite undirected graph; see below. }.
	We often drop subscripts if these are clear from context. For two digraphs $H_1$ and $H_2$, the union $H_1 \cup H_2$ is the digraph with vertex set $V(H_1) \cup V(H_2)$ and edge set $E(H_1) \cup E(H_2)$. 
	We say an undirected graph $G$ is bipartite with bipartition $A,B$ if $V(G)=A\cup B$ and $E(G) \subseteq \{ab: a \in A, \; b \in B \}$.

	For a digraph $G$ and $v\in V(G)$, we denote the set of outneighbours and inneighbours of $v$ by $N_G^{+}(v)$ and $N_G^{-}(v)$, respectively, and we write $d_G^{+}(v):=|N_G^{+}(v)|$ and $d_G^{-}(v):=|N_G^{-}(v)|$ for the out- and indegree of~$v$, respectively.  
	For $S\subseteq V(G)$ we write $d^{-}(v,S):= |N_G^-(v) \cap S|$ and $d^{+}(v,S):= |N_G^+(v) \cap S|$. 
	We write $\delta^+(G)$ and $\delta^-(G)$, respectively, for the minimum out- and indegree of $G$ and $\delta^{0}(G):=\min\{\delta^{+}(G),\delta^{-}(G)\}$ for the minimum semi-degree.
	Similarly, the maximum semi-degree~$\Delta^{0}(G)$ of~$G$ is defined by $\Delta^{0}(G):=\max\{\Delta^{+}(G),\Delta^{-}(G)\}$ where $\Delta^{+}(G)$ and $\Delta^{-}(G)$ denote the maximum out- and maximum indegree of~$G$, respectively.
	A digraph is called \textit{$d$-regular} if each vertex has exactly $d$ outneighbours and $d$ inneighbours. 
	
	The notation above extends to undirected graphs in the obvious ways. 
	In particular for an undirected graph~$G$, we write $\Delta(G)$ and $\delta(G)$, respectively for the maximum degree and the minimum degree. 
	A graph is called $d$-regular if each vertex has exactly $d$ neighbours. For a vertex~$v \in V(G)$ and subset $S \subseteq V(G)$, we write $d_G(v,S):=|N_G(v) \cap S|$.

	A directed path~$Q$ in a digraph~$G$ is a subdigraph of~$G$ such that $V(Q) = \{v_1, \ldots, v_k \}$ for some~$k \in \mathbb{N}$ and  $E(Q) = \{v_1v_2, v_2v_3, \ldots, v_{k-1}v_k \}$. We denote such a directed path by its vertices in order, i.e.\ we write $Q = v_1v_2 \cdots v_k$.
	A directed cycle in $G$ is exactly the same except that it also includes the edge $v_kv_1$. Sometimes we identify paths with their edge sets.

	A set of vertex-disjoint directed paths $\mathcal{Q}=\{Q_1,Q_2,\ldots\}$ in a digraph $G$ is called a \textit{path system} in~$G$. We interchangeably think of $\mathcal{Q}$ as a set of vertex-disjoint directed paths in $G$ and as a subdigraph of $G$ with vertex set $V(\mathcal{Q}) = \bigcup_{i} V(Q_i)$ and edge set $E(\mathcal{Q}) = \bigcup_{i} E(Q_i)$. We sometimes call this subdigraph the graph induced by~$\mathcal{Q}$.  
	A matching~$M$ in a digraph (or undirected graph)~$G$ is a set of edges $M \subseteq E(G)$ such that every vertex of~$G$ is incident to at most one edge in~$M$. 
	If $M = \{a_ib_i: i \in[m]\}$ is a matching in a digraph, then we write $V^+(M):=	\{a_i : i \in [m]\}$ and $V^-(M) = \{ b_i : i \in [m] \}$. A \textit{$1$-factor} in a digraph $G$ is a set of vertex-disjoint directed cycles whose union has the same vertices as $G$.

	Throughout the paper, we will work with partitions of a vertex set $V$ of the form
	$\{V_{ij}: i,j \in [k]\}$ for some $k \in \mathbb{N}$ (so a partition into $k^2$ parts). 
	For each $i$, we write $V_{i*} := V_{i1} \cup \cdots \cup  V_{ik}$ and $V_{*i} := V_{1i} \cup \cdots \cup V_{ki}$. Note that $\{V_{i*}: i \in [k]\}$ is a partition of $V$ as is $\{V_{*i}: i \in [k]\}$. 
	Note also that $V_{ij} = V_{i*} \cap V_{*j}$.

	For two sets $A$ and $B$, the \textit{symmetric difference} of $A$ and $B$ is the set $A\triangle B := (A\setminus B) \cup (B\setminus A)$.
	For $x, y \in (0,1]$, we often use the notation $x \ll y$ to mean that $x$ is sufficiently small as a function of $y$ i.e.\ $x \leq f(y)$ for some implicitly given non-decreasing function $f:(0,1] \rightarrow (0,1]$. 
	We implicitly assume all constants in such hierarchies are positive, and we omit floors and ceilings whenever this does not affect the argument.

	We use the following non-standard notation. 
	Let $G$ be a directed graph and let $U,W \subseteq V(G)$ not necessarily disjoint. 
	We define $G[U,W]$ to be the auxiliary undirected bipartite graph with bipartition $U , W$ where, for each $u \in U$ and $w \in W$, $uv$ is an (undirected) edge of $G[U,W]$ if and only if $uv$ is a directed edge in $E(G)$. Note that for each vertex in $U\cap W$, there are two copies of the vertex in $G[U,W]$ which are viewed as distinct. So $G[U,W]$ has $|U| + |W|$ vertices and $e_G(U,W)$ edges.

	\subsection{Sketch of proof of Theorem~\ref{thm:main}}
	
	Theorem~\ref{thm:main} is proved in several steps. One of the key ingredients is  a structural result for directed graphs (Theorem~\ref{thm:digraphstructure}) that we derive from a result of K{\"u}hn, Lo, Osthus and Staden~\cite{IntoTwoBipartiteExpander} 
	about partitioning dense regular undirected graphs into robust expanders. Robust expansion is a notion introduced and used by K{\"u}hn and Osthus together with several coauthors to obtain a number of breakthrough results on (di)graph decompositions and Hamiltonicity (see e.g.~\cite{RobustExpanderHamilton, On-Kelly-Conjecture, IntoTwoBipartiteExpander}). In Section~\ref{sec:Robust}, we give the necessary background on robust expansion before proving Theorem~\ref{thm:digraphstructure}, which we informally describe below.

	Informally, robust expanders are dense (di)graphs that are highly connected in a certain  sense, and one of their key properties is that they are Hamiltonian under suitable (mild) degree conditions (see Theorem~\ref{thm:RobustExpanderImpliesHamilton}). Moreover, they are robust to small alterations (see Lemma~\ref{lma:birobexpdifferent}). If we could show that every $d$-regular digraph (resp. oriented graph) can be partitioned into at most $n/(d+1)$ (resp. $n/(2d+1)$) robust expanders, it would be enough to prove Theorem~\ref{thm:main}. Such a partition does not exist in general, but our structural result, Theorem~\ref{thm:digraphstructure}, gives us a starting point. Roughly, it says that for any $d$-regular $n$-vertex digraph $G$ with $d$ linear in $n$ and $n$ sufficiently large, there exist two vertex partitions $V(G)=V_{1*} \cup \cdots \cup V_{k*}$ and $V(G)=V_{*1} \cup \cdots \cup V_{*k}$ with $k\leq 1+n/(d+1)$ such that for each~$i\in[k]$, $G[V_{i*}, V_{*i}]$ is a bipartite robust expander and $|V_{i*}|\approx |V_{*i}|$.
	Note here that $k$, the number of parts in each partition, is at most one more than the number of cycles we desire (in the case of digraphs).
	Writing $V_{ij} = V_{i*} \cap V_{*j}$ for all $i,j\in[k]$, we can think of our two partitions as a single $k^2$-partition $\mathcal{P}=\{V_{ij}: i,j \in [k]\}$ of $V(G)$.
	
	If the partition $\mathcal{P}$ described above is balanced, meaning that $|V_{i*}| = |V_{*i}|$ for all $i$, then it turns out that $G$ can be covered by $k$ vertex-disjoint cycles.\footnote{This follows essentially from Lemma~\ref{lma:longcycle2} although we do not show it explicitly in the paper.} 
	So to prove Theorem~\ref{thm:main}, we would like to ensure our partition is balanced, and when $k = 1+ \lfloor n/(d+1) \rfloor$, we would like to slightly improve on the number of cycles.
	For the latter, in Section~\ref{sec:longcycles}, we define an auxiliary graph $S(\mathcal{P})$ for the partition~$\mathcal{P}$, which has vertex set $[k]$ and $ij\in E(S(\mathcal{P}))$ if and only if $V_{ij}\cup V_{ji}\neq \emptyset$. 
	For each connected component $I \subseteq [k]$ of $S(\mathcal{P})$, Lemma~\ref{lma:longcycle2} shows how to find a cycle whose vertices are exactly $\bigcup_{i \in I} V_{i*} \cup V_{*i}$ (provided $\mathcal{P}$ is balanced). These cycles are necessarily disjoint for different connected components of $S(\mathcal{P})$, so we obtain a cycle partition where the number of cycles is exactly the number of connected components of $S(\mathcal{P})$.
	For digraphs, if $S(\mathcal{P})$ has at least one edge (so has at most $k-1 \leq n/(d+1)$ connected components), then Theorem~\ref{thm:main} follows. Similarly, for oriented graphs, one can show (using  Proposition~\ref{prop:orientedcomponent}) that at least one edge in $S(\mathcal{P})$ is enough to prove Theorem~\ref{thm:main} for oriented graphs.
	Thus, it is enough if our partition $\mathcal{P}$ is balanced and $S(\mathcal{P})$ has at least one edge.

	To ensure our partition is balanced, i.e.\ that $|V_{i*}| = |V_{*i}|$ for all $i$, note first that, as mentioned earlier, Theorem~\ref{thm:digraphstructure} already guarantees that $|V_{i*}|\approx |V_{*i}|$ for all $i$. The idea will be to find and contract a suitable collection of vertex-disjoint paths $\mathcal{Q}$.
	Here suitable simply means that the numbers of edges of $\mathcal{Q}$ between the various parts of $\mathcal{P}$ are related in the right way, in which case we call $\mathcal{Q}$ a $\mathcal{P}$-balanced path system (introduced in Section~\ref{sec:Path-Contraction}).
	If we can find such a $\mathcal{P}$-balanced path system with few edges, then contracting the paths will result in a graph with an adjusted partition~$\mathcal{P}'$ that is balanced and has similar properties as $\mathcal{P}$ (Proposition~\ref{prop:contracting}), and any cycle partition in the contracted graph can be lifted to a cycle partition in the original graph (Propositions~\ref{prop:1-fact}). If the path system $\mathcal{Q}$ satisfies some further properties then we can also guarantee that $S(\mathcal{P}')$ has at least one edge and this will be enough to prove Theorem~\ref{thm:main}.
	Lemma~\ref{lma:balancingpathsystemcombined} says that we can always find a $\mathcal{P}$-balanced path system with the required properties, and so we reduce the task of proving Theorem~\ref{thm:main} to the task of proving Lemma~\ref{lma:balancingpathsystemcombined}.

	In Section~\ref{sec:pfbalancingpathsystem} we reduce the proof of Lemma~\ref{lma:balancingpathsystemcombined} to two lemmas, namely  Lemmas~\ref{lem:balancedpath} and~\ref{lma:regdigraph}. Lemma~\ref{lem:balancedpath} allows us to find a $\mathcal{P}$-balanced path system with the desired properties under most circumstances, while Lemma~\ref{lma:regdigraph} (together with Proposition~\ref{prop:balancepartition}) allows us to find such a path system in the remaining ``extremal'' circumstances when $G$ is close to the disjoint union of cliques.
	Lemma~\ref{lem:balancedpath} is proved in Section~\ref{sec:pf_balancedpath} using a flow argument 
	and Lemma~\ref{lma:regdigraph} is proved in Section~\ref{sec:pf_regdigraph}. 
	These last two lemmas are the most technical parts of the paper, so
	we defer their sketch of proofs to their respective sections.

	\section{Robust expanders}	\label{sec:Robust}
	
	We first define robust expansion for graphs. 
	Let $0 < \nu \le \tau < 1$ and let $G$ be a graph on $n$ vertices. 
	For $S \subseteq V(G)$, the \emph{$\nu$-robust neighbourhood} $RN_{\nu,G}(S)$ is the set of all those vertices with at least $\nu n$ neighbours in~$S$. 
	We say that $G$ is a robust $(\nu, \tau)$-expander if every $S \subseteq V(G)$ with $\tau n \le |S| \le (1 - \tau) n$ satisfies $|RN_{\nu,G} (S)| \ge |S| + \nu n$. In fact we will mainly be concerned with bipartite robust expanders.
	Let $G$ be a bipartite graph with bipartition~$A, B$.
	We say that $G$ is a \emph{bipartite robust $(\nu, \tau)$-expander with bipartition~$A, B$} if every $S \subseteq A$ with $\tau |A| \le |S| \le (1 - \tau) |A|$ satisfies $|RN_{\nu,G} (S) \cap B| \ge |S| + \nu n$. 
	Note that the order of $A$ and $B$ matters here.
	Our first lemma says that bipartite robust expansion is robust to small alterations; the lemma can easily be derived from the definition of robust expansion.
	
	\begin{lemma}[{\cite[Lemma~3.4.9]{StadenThesis}}] \label{lma:birobexpdifferent}
		Let $0 < 1/n \ll  \nu \le \tau \ll 1$ with $\nu \le 1/2$.
		Let $G$ be a bipartite graph with $U \subseteq V(G)$. 
		Suppose $G[U]$ is a bipartite robust $(\nu, \tau)$-expander on $n$ vertices with bipartition~$A,B$ 
		and that $A',B' \subseteq V(G)$  are sets satisfying $|A\triangle A'|+|B \triangle B'|\leq \nu |A|/4$. 
		Then $G[ A' \cup B' ]$ is a bipartite robust $(\nu/2, 2\tau)$-expander with bipartition~$A',B'$.
	\end{lemma}

	Next we give a structural result due to K{\"u}hn, Lo, Osthus and Staden~\cite{IntoTwoBipartiteExpander} which states that any regular graph of linear minimum degree has a vertex partition into a small number of parts where each part induces a robust expander or a bipartite robust expander. 
	In fact we state the special case of this result for bipartite graphs, which is all we require.

	\begin{theorem}[{Bipartite special case of~\cite[Theorem~3.1]{IntoTwoBipartiteExpander}}]
		\label{thm:structure}
		For all $\alpha, \tau > 0$ and every non-decreasing function $f: (0, 1) \rightarrow (0, 1)$, there exists $n_0$ such that the following holds.
		For all $d$-regular bipartite graphs $G$ on $2n \geq n_0$ vertices with bipartition $A,B$
		and $d \ge \alpha n$, there exist $\rho, \nu$ with
		\begin{align*}
			1/n_0 \le \rho \le \nu \le \tau; \qquad \rho \le f(\nu); \qquad \text{and} \qquad 1/n_0 \le f(\rho)
		\end{align*}
		such that there is a partition of~$V(G)$ into sets $A_1,\dots, A_k,B_1, \dots, B_k$ with the following properties for all $i \in [k]$:
		\begin{enumerate}[label={\rm (\roman*)}]
			\item \label{itm:structure1} $G[A_i \cup B_i]$ is a bipartite robust $(\nu, \tau)$-expander with bipartition $A_i,B_i$;
			\item \label{itm:structure2} for all $x \in A_i \cup B_i$ and $j \in [k]$, $d(x, A_i \cup B_i) \ge d(x, A_j \cup B_j) $, so  in particular, $\delta( G[ A_i  \cup B_i ] ) \ge d/k$;
			\item \label{itm:structure3} $\left| |A_i| - |B_i| \right| \le 2 \rho n$;
			\item \label{itm:structure4} all but at most $2 \rho n $ vertices $x \in A_i \cup B_i$ satisfy $d(x, A_i \cup B_i) \ge d - 2 \rho n $;
			\item  \label{itm:structure5} $k \le   n / (d-2 \rho n) $;
			\item \label{itm:structure6} $A_1, \ldots, A_k$ is a partition of $A$ and $B_1, \ldots, B_k$ is a partition of $B$.
		\end{enumerate}
	\end{theorem}

	\begin{remark}
		\label{rem:remark1}
		Note that \ref{itm:structure1}--\ref{itm:structure4} follow directly from the statement of~{\rm\cite[Theorem~3.1]{IntoTwoBipartiteExpander}}. While \ref{itm:structure5} and \ref{itm:structure6} cannot be deduced immediately from the statement of~{\rm\cite[Theorem~3.1]{IntoTwoBipartiteExpander}}, they follow immediately from the proof.\footnote{
			The idea of the proof is to successively refine partitions of $V(G) = A \cup B$ as follows. 
			Assume we have obtained a partition $\mathcal{U} = \{U_1, \ldots, U_r\}$ of~$V(G)$ such that there are few edges leaving or entering~$U$ for every $U \in \mathcal{U}$.
			If for some $U \in \mathcal{U}$ we have that $G[U]$ is not a bipartite robust expander (with bipartition $U \cap A, U \cap B$), then there is some $S \subseteq U \cap A$ whose robust neighbourhood $R \subseteq U \cap B$ is not much larger than~$S$.
			Writing $U' = S \cup R$ and $U'' = U \setminus U'$, we let $\mathcal{U}' = \mathcal{U} \setminus \{U\} \cup \{U', U''\}$.
			It is not too hard to show (using the fact that $G$ is regular) that, as with $\mathcal{U}$, there are not many edges entering or leaving each $U \in \mathcal{U}'$.
			We continue refining the partition in this way until we obtain a partition~$\mathcal{U}^*$ where every $U \in \mathcal{U}^*$ satisfies that $G[U]$ is a bipartite robust expander with bipartition $U \cap A, U \cap B$.
			The process of refining the partition must eventually stop because each $U \in \mathcal{U}$ cannot be much smaller than~$d$ (the degree of $G$) since not many edges enter of leave~$U$.
			This essentially shows \ref{itm:structure1}, \ref{itm:structure3}, \ref{itm:structure5} and \ref{itm:structure6}, while \ref{itm:structure2} and \ref{itm:structure4} are obtained by making slight adjustments to the final partition.}
	\end{remark}

	We now define robust expansion for digraphs. 
	Let $0 < \nu \le \tau < 1$ and 
	let $G$ be a digraph on $n$ vertices. 
	For $S \subseteq V(G)$, the \emph{$\nu$-robust outneighbourhood} $RN^+_{\nu,G}(S)$ is the set of all those vertices with at least $\nu n$ inneighbours in~$S$. 
	We say that $G$ is a \textit{robust $(\nu, \tau)$-outexpander} if every $S \subseteq V(G)$ with $\tau n \le |S| \le (1 - \tau) n$ satisfies $|RN^+_{\nu,G} (S)| \ge |S| + \nu n$.

	\begin{proposition}
		\label{pr:contractbipartite}
		Let $0 < 1/m \ll  \nu \le \tau \ll 1$ with $\nu \le 1/2$. 
		Suppose $G$ is a bipartite graph with bipartition~$A,B$, where $A=\{a_1, \ldots, a_m\}$ and $ B=\{b_1, \ldots, b_m\}$. If $G$ is a bipartite robust $(\nu, \tau)$-expander with bipartition~$A,B$,
		then the digraph $H$ with vertex set $A$ and edge set $E(H) = \{a_ia_j: a_ib_j \in E(G), i \not= j\}$ is a robust $(\nu/2, 2\tau)$-outexpander. 
	\end{proposition}
	The proposition above follows immediately from the definitions; we crudely replace $(\nu, \tau)$ with $(\nu/2, 2\tau)$ to account for the loss of any edges of the form $a_ib_i$.

	We now state and prove a structure lemma for regular digraphs that is derived from Theorem~\ref{thm:structure}. This will be one of the key ingredients in the proof of Theorem~\ref{thm:main}. It says that any dense regular digraph~$G$ has two vertex partitions $V(G)=V_{1*} \cup \cdots \cup V_{k*}$ and $V(G)=V_{*1} \cup \cdots \cup V_{*k}$ with $k$ relatively small such that for each~$i$, the (undirected) bipartite graph $G[V_{i*}, V_{*i}]$ is a bipartite robust expander (with bipartition $V_{i*}, V_{*i}$). 
	Various other degree and size conditions relating to the partition are also given.
	Note that in the theorem below, we actually give a partition $\{V_{ij}: i,j \in [k]\}$ and recall that for each $i\in [k]$, we write $V_{i*} := V_{i1} \cup \cdots \cup  V_{ik}$ and $V_{*i} := V_{1i} \cup \cdots \cup V_{ki}$. 
	Thus $V_{ij} = V_{i*} \cap V_{*j}$.
	
	\begin{theorem} \label{thm:digraphstructure}
		For all $\alpha, \tau > 0$ and every non-decreasing function $f : (0, 1) \rightarrow (0, 1)$, there exists $n_0$ such that the following holds.
		For all $d$-regular digraphs~$G$ on $n \ge n_0$ vertices with $d \ge \alpha n$, there exist $\rho, \nu$ with 
		\begin{align*}
			1/n_0 \le \rho \le \nu \le \tau; \qquad \rho \le f(\nu); \qquad \text{and} \qquad 1/n_0 \le f(\rho)
		\end{align*}
		such that there is a partition~$\mathcal{P} = \{ V_{ij}: i,j \in [k]\}$ of~$V(G)$ satisfying, for all $i,j \in [k]$,
		\begin{enumerate}[label={\rm(\roman*)}]
			\item \label{itm:di1}$G[V_{i*},V_{*i}]$ is a bipartite robust $(\nu, \tau)$-expander with bipartition  $V_{i*},V_{*i}$;
			\item \label{itm:di2} for all $x \in V_{ij}$ and $i',j' \in [k]$, $d^+(x, V_{*i}) \ge d^+(x, V_{*i'})$ and $d^-(x, V_{j*}) \ge d^-(x, V_{j'*})$, so in particular, $\delta( G [ V_{i*},V_{*i} ] ) \ge d/k$;
			\item \label{itm:di3} $\left| |V_{i*}| - |V_{*i}| \right| \le 2\rho n$;
			\item \label{itm:di4} all but at most $2\rho n $ vertices $x \in V_{ij}$ satisfy $d^+(x, V_{*i}), d^-(x, V_{j*})  \ge d - 2\rho n $;
			\item \label{itm:di5} $k \le   n / (d-2\rho n)$.
		\end{enumerate}
	\end{theorem}

	\begin{proof}
		We simply apply Theorem~\ref{thm:structure} to the obvious bipartite graph obtained from the directed graph $G$, with the natural correspondence in parameters, as follows. 
		
		For any $d$-regular digraph $G$ on $n$ vertices, let $H_G$ be the $d$-regular bipartite undirected graph on $2n$ vertices defined as follows.
		Let $V(H_G) = A \cup B$, where $A$ and $B$ are disjoint copies of~$V(G)$ and for $a \in A$ and $b \in B$ let $ab \in E(H_G)$ if and only if $ab \in E(G)$.
		Now applying Theorem~\ref{thm:structure} with $H_G$ playing the role of~$G$, we obtain the following statement.

		For all $\alpha, \tau > 0$ and every non-decreasing function $f: (0, 1) \rightarrow (0, 1)$, there exists $n_0$ such that the following holds.
		For all $d$-regular digraphs~$G$ on $n \geq n_0$ vertices and $d \ge \alpha n$, taking $H_G$ to be the corresponding bipartite $d$-regular graph on $2n \geq n_0$ vertices there exist $\rho, \nu$ with
		\begin{align*}
			1/n_0 \le \rho \le \nu \le \tau; \qquad \rho \le f(\nu) ;\qquad \text{and} \qquad 1/n_0 \le f(\rho)
		\end{align*}
		such that there is a partition of $V(H_G)$ into sets $A_1,\dots, A_{k}, B_1, \dots, B_{k}$ with the following properties for all $i \in [k]$:
		\begin{enumerate}[label={\rm (\roman*$'$)}]
			\item \label{itm:1'} $H_G[A_i \cup B_i]$ is a bipartite robust $(\nu, \tau)$-expander with bipartition $A_i,B_i$;
			\item \label{itm:2'} for all $x \in A_i \cup B_i$ and $j \in [k]$, $d_{H_G}(x, A_i \cup B_i) \ge d_{H_G}(x, A_j \cup B_j) $, so in particular, $\delta( H_G [ A_i  \cup B_i ] ) \ge d/k$;
			\item \label{itm:3'} $\left| |A_i| - |B_i| \right| \le 2 \rho n$;
			\item \label{itm:4'} all but at most $2 \rho n $ vertices $x \in A_i \cup B_i$ satisfy $d_{H_G}(x, A_i \cup B_i) \ge d - 2 \rho n $;
			\item \label{itm:5'} $k \le   n / (d-2 \rho n)$;
			\item  \label{itm:6'}     $A_1, \ldots, A_k$ is a partition of $A$ and $B_1, \ldots, B_k$ is a partition of $B$.
		\end{enumerate}
		For $i,j \in [k]$, define $V_{ij} = A_i \cap B_j$ (where we think of $A_i$ and $B_j$ as sets of vertices of the digraph~$G$).  First note that, by~\ref{itm:6'},  $\mathcal{P} = \{V_{ij} : i,j \in [k] \}$ is a partition of $V(G)$, and $A_i = V_{i*}$ and $B_i = V_{*i}$.
		Now the natural correspondence between $H_G$ and $G$ means that \ref{itm:1'}--\ref{itm:5'} imply~\ref{itm:di1}--\ref{itm:di5}, respectively.
	\end{proof}


	\section{Finding long cycles} \label{sec:longcycles}
	
	Theorem~\ref{thm:digraphstructure} from the previous section shows that every (dense) regular digraph has a vertex partition with some useful properties. In this section we show how properties \ref{itm:di1} and \ref{itm:di2} from Theorem~\ref{thm:digraphstructure} together with a simple balancing condition on the partition allows us to construct few long cycles that can cover all the vertices of several parts in the partition.
	
	Let $G$ be a digraph on~$V$ and
	let $\mathcal{P} =\{V_{ij}:i,j \in [k]\}$ be a partition of~$V$ where we allow some parts to be empty. 
	We say that $\mathcal{P}$ is \emph{balanced} if $|V_{i*}| = |V_{*i}|$ for all $i \in [k]$. For $i, j \in [k]$, let $G_{ij}$ be the subdigraph of~$G$ on $V_{i*} \cup V_{*j}$ with edges from~$V_{i*}$ to~$V_{*j}$, that is,  $E(G_{ij}) = E_G(V_{i*},V_{*j}) = E(G) \cap (V_{i*} \times V_{*j})$.

	Define $S(\mathcal{P})$ to be the graph (without loops) on $[k]$ such that for all $i,j\in[k]$ with $i\neq j$, $ij \in E(S(\mathcal{P}))$ if and only if $V_{ij} \cup V_{ji} \ne \emptyset$ . The connected components of $S(\mathcal{P})$ will determine which parts of $\mathcal{P}$ can be covered by one long cycle. We remark that, by definition,  $\bigcup_{i\in I}(V_{i*}\cup V_{*i})$ and $\bigcup_{j\in J}(V_{j*}\cup V_{*j})$ are disjoint for two distinct connected components $I$ and $J$ of $S(\mathcal{P})$.
	The aim of this section is to prove the following lemma.
	
	\begin{lemma} \label{lma:longcycle2}
		Let $1/m \ll \nu \leq \tau \ll \alpha <1$.
		Let $G$ be a digraph with a balanced vertex partition~$\mathcal{P} = \{V_{ij} : i,j \in [k]\}$.
		Let $I$ be a connected component in~$S(\mathcal{P})$.
		Suppose that, for all~$i \in I$, $G[V_{i*},V_{*i}]$ is a bipartite robust $(\nu, \tau)$-expander with bipartition $V_{i*},V_{*i}$ such that $|V_{i*}| \ge m$ and $\delta( G[V_{i*},V_{*i}] ) \ge \alpha |V_{i*}|$. 
		Then there exists a cycle~$C$ in~$G$ with $V(C) = \bigcup_{i \in I} \left( V_{i*} \cup V_{*i} \right)$. 
	\end{lemma}
	
	The lemma above will be used as follows.  Suppose we have a dense $d$-regular digraph $G$ with a vertex partition $\mathcal{P}$ such as that given by Theorem~\ref{thm:digraphstructure} but with the additional property that $\mathcal{P}$ is balanced. Then Lemma~\ref{lma:longcycle2} applied to each connected component of $S(\mathcal{P})$ gives us a collection of $s$ vertex-disjoint cycles that cover~$V(G)$, where $s$ is the number of connected components in~$S(\mathcal{P})$. 
	So in later sections we will be interested in obtaining balanced partitions~$\mathcal{P}$ where the number of connected components of $S(\mathcal{P})$ is ``small''.
	
	We need the following theorem, which states that a robust outexpander with linear minimum degree contains a (directed) Hamilton cycle. 
	
	\begin{theorem}[\cite{RobustExpanderHamilton}; see also~\cite{LoPatel}]
		\label{thm:RobustExpanderImpliesHamilton}
		Let $1/n \ll \nu\leq \tau\ll\gamma<1$.
		Let $G$ be a robust $(\nu,\tau)$-outexpander on $n$ vertices with $\delta^{0}(G)\geq \gamma n$.
		Then $G$ contains a Hamilton cycle.
	\end{theorem}
	
	The next lemma shows that, under the conditions of Lemma~\ref{lma:longcycle2}, $V_{i*} \cup V_{*i}$ can be covered either by vertex-disjoint paths from $V_{i*} \setminus V_{ii}$ to $V_{*i} \setminus V_{ii}$ or by a cycle, and it is used inductively to prove Lemma~\ref{lma:longcycle2}.

	\begin{lemma} \label{lma:longcycle1}
		
		Let $1/m \ll \nu \leq \tau \ll \alpha <1$.
		Let $G$ be a digraph with a vertex partition~$\mathcal{P} = \{V_{ij} : i,j \in [k]\}$.
		Let $i \in [k]$.
		Suppose that $G[V_{i*},V_{*i}]$ is a bipartite robust $(\nu, \tau)$-expander with bipartition $V_{i*},V_{*i}$ such that $|V_{i*}| = |V_{*i}| = m $ and $\delta( G[V_{i*},V_{*i}] ) \ge \alpha m$.
		Let $\phi: V_{i*} \setminus V_{ii} \rightarrow  V_{*i} \setminus V_{ii}$ be a bijection.
		Then there exists a path system~$\mathcal{Q}$ in~$G_{ii}$ such that $\mathcal{Q} \cup \{\phi(v)v : v \in V_{i*} \setminus V_{ii} \}$ is a cycle with vertex set~$V_{i*} \cup V_{*i}$. 
		In the special case that $V_{i*} = V_{*i} = V_{ii}$, we have that $\phi$ is an empty function and $\mathcal{Q}$ is a cycle (rather than a path system) with vertex set $V_{ii}$.
	\end{lemma}
	
	\begin{proof}
		Let $H$ be the digraph on~$V_{i*}$ obtained from~$G_{ii}$ by identifying~$v$ with~$\phi(v)$ for all $v \in V_{i*} \setminus V_{ii}$ and deleting any loops.
		Note that $\delta^0(H) \ge \delta( G[V_{i*},V_{*i}] )  - 1 \ge \alpha m /2$ and $H$ is a robust $(\nu/2, 2\tau)$-outexpander by Proposition~\ref{pr:contractbipartite} applied to the undirected bipartite graph $G[V_{i*},V_{*i}]$. 
		By Theorem~\ref{thm:RobustExpanderImpliesHamilton}, $H$ has a Hamilton cycle~$C$. 
		Note that if $V_{i*}=V_{ii}=V_{*i}$, then $H$ coincides with $G[V_{ii}]$,  proving the lemma in this case. 
		If not, let $V_{i*}\setminus V_{ii} = \{ x_1, \dots, x_{\ell}\}$ be such that $C$ visits $x_1, \ldots, x_{\ell}$ in that order.
		Thus, $C$ can be decomposed into paths $P_1, \dots, P_{\ell}$ such that $P_j$ is a path from~$x_j$ to~$x_{j+1}$, where we take $x_{\ell+1} = x_1$.
		Let $Q_j$ be obtained from~$P_j$ by replacing~$x_{j+1}$ with~$\phi(x_{j+1})$.
		Note that $Q_j$ is a path in $G_{ii}$ from~$x_j$ to~$\phi(x_{j+1})$. 
		The result follows by setting 
		$\mathcal{Q} = \{Q_1, \ldots, Q_{\ell}\}$.
	\end{proof}

	We now prove Lemma~\ref{lma:longcycle2}. 
	For a connected component $I$ in $S(\mathcal{P})$, 
	the idea is to apply Lemma~\ref{lma:longcycle1} to each $i \in I$. 
	Some care is needed to ensure that the union of path systems forms only one cycle. 
	
	\begin{proof}[Proof of Lemma~\ref{lma:longcycle2}] If $I$ consists of a single vertex say~$i$, then the result follows by Lemma~\ref{lma:longcycle1} (since in that case $V_{i*} = V_{*i} = V_{ii}$). Now assume $|I|\geq 2$, so $V_{i*}\cup V_{*i}\neq V_{ii}$ for all $i\in I$.
		Without loss of generality, let $I = [\ell]$ 
		and order the indices of $I$ such that for each $j \in [\ell-1]$, there is a $j' \in [\ell] \setminus [j]$ with $j j' \in E(S(\mathcal{P}))$. This can be achieved since $I$ is connected, e.g.\ by taking a reverse breadth-first search ordering. Note that $V_{ij}=V_{ji}=\emptyset$ for all $i\in[\ell]$ and $j\notin [\ell]$.

		For $i\in[\ell]$, let $W^+_i=\bigcup_{i'\in [i],\, j'\notin[i]}V_{i'j'} = \bigcup_{i' \in [i]}V_{i'*} \setminus \bigcup_{i',j'\in [i]}V_{i'j'}$ and $W^-_i=\bigcup_{i'\notin [i],\, j'\in[i]}V_{i'j'} = \bigcup_{j' \in [i]}V_{*j'} \setminus \bigcup_{i',j'\in [i]}V_{i'j'}$.
		Since $\mathcal{P}$ is balanced, $|W^+_i| = |W^-_i|$ for each $i \in [\ell]$. 
		Also, by our ordering of the indices, we have that $W_i^+,W_i^-\neq \emptyset$ for $i\in[\ell-1]$ and that $W_{\ell}^+=W_{\ell}^-= \emptyset$.

		Let $\mathcal{Q}_0=W_0^+=W_0^-=\emptyset$ and suppose for some $i \in [\ell]$, we have already found a path system~$\mathcal{Q}_{i-1}$ such that $V(\mathcal{Q}_{i-1}) = \bigcup_{i' \in [i-1]} \left( V_{i'*} \cup V_{*i'} \right)$ and $\mathcal{Q}_{i-1}$ consists of precisely $|W^+_{i-1}|$ paths from~$W^+_{i-1}$ to~$W^-_{i-1}$. 
		We now construct $\mathcal{Q}_{i}$ as follows. 
		Let $\phi: V_{i*} \setminus V_{ii}  \rightarrow  V_{*i} \setminus V_{ii} $ be any bijection such that if there is a path in~$\mathcal{Q}_{i-1}$ from~$v_+ \in V_{*i}\setminus V_{ii}$ to~$v_- \in V_{i*}\setminus V_{ii}$, then $\phi(v_-) = v_+ $.
		Apply Lemma~\ref{lma:longcycle1} and obtain a path system~$\mathcal{Q}_i'$ in~$G_{ii}$ such that $\mathcal{Q}_i' \cup \{\phi(v)v : v \in V_{i*}\setminus V_{ii}\}$ is a cycle with vertex set~$V_{i*} \cup V_{*i}$.
		We set $\mathcal{Q}_{i} = \mathcal{Q}_{i-1} \cup \mathcal{Q}_i'$.

		Suppose $i = \ell$.
		Note that $W^-_{\ell-1} = V_{\ell*} \setminus V_{\ell\ell}$ and $W^+_{\ell-1} = V_{*\ell} \setminus V_{\ell\ell}$.
		Hence $\mathcal{Q}_{\ell-1}$ consists of paths from~$\phi(v) \in W_{\ell-1}^+$ to~$v \in W_{\ell-1}^-$, we have $V(\mathcal{Q}_{\ell-1}) = \bigcup_{i' \in [\ell-1]} \left( V_{i'*} \cup V_{*i'} \right)= \bigcup_{i' \in [\ell]} \left( V_{i'*} \cup V_{*i'} \right) \setminus V_{\ell \ell}$. 
		Since $\mathcal{Q}_{\ell}' \cup \{\phi(v)v : v \in V_{\ell *} \setminus V_{\ell \ell}\}$ is a cycle with vertex set~$V_{\ell *} \cup V_{* \ell}$, we deduce that $\mathcal{Q}_{\ell} = \mathcal{Q}_{\ell}' \cup \mathcal{Q}_{\ell-1}$ is a cycle with vertex set~$\bigcup_{i' \in [\ell]} \left( V_{i'*} \cup V_{*i'} \right)$ as required.

		Suppose $i \in [\ell - 1]$.
		Note that $V(\mathcal{Q}_{i}) = \bigcup_{i' \in [i]} \left( V_{i'*} \cup V_{*i'} \right)$.
		It remains to check that $\mathcal{Q}_i$ is a path system consisting of precisely $|W_i^+|$ paths from $W_i^+$ to $W_i^-$.
		Note first that paths in $\mathcal{Q}'_i \subseteq G_{ii}$ start in~$V_{i*} \setminus V_{ii}$ and end in~$V_{*i} \setminus V_{ii}$ and have all internal vertices in~$V_{ii}$. 
		On the other hand a path in~$\mathcal{Q}_{i-1}$ can only intersect $V_{*i}$ at its start point and $V_{i*}$ at its end point and it avoids~$V_{ii}$.
		Therefore in $\mathcal{Q}_{i} = \mathcal{Q}_{i-1} \cup \mathcal{Q}'_i$, all indegrees and outdegrees are at most $1$. A vertex has indegree zero in~$\mathcal{Q}_{i-1}\cup \mathcal{Q}_i'$ if and only if it is a start point of some path in $\mathcal{Q}_{i-1}\cup \mathcal{Q}_i'$ but not an endpoint of any path in $\mathcal{Q}_{i-1}\cup \mathcal{Q}_i'$, i.e.~the set of vertices of indegree zero is
		\begin{align*}
			\left(W_{i-1}^+\cup (V_{i*}\setminus V_{ii})\right)\setminus   \left(W_{i-1}^-\cup (V_{*i}\setminus V_{ii})\right)=\left(W_{i-1}^+\cup V_{i*}\right)\setminus V_{*i}
			=W_i^+
		\end{align*}
		and similarly the set of vertices of outdegree zero is $W_i^-$. Finally it remains to check that there are no cycles in~$\mathcal{Q}_{i-1}\cup \mathcal{Q}_i'$. By our choice of $\phi$, if there is a cycle, it spans all the vertices of~$V(\mathcal{Q}_{i-1}\cup \mathcal{Q}_i')$, but this cannot happen because vertices in $W_i^+\neq \emptyset$ have in-degree zero.
	\end{proof}


	\section{Balanced path systems and path contraction}\label{sec:Path-Contraction}

	As mentioned in the previous section, in order to apply Lemma~\ref{lma:longcycle2} to obtain a suitable cycle partition of a dense regular digraph $G$, we will require a vertex partition $\mathcal{P}$ of $G$ such as that given by Theorem~\ref{thm:digraphstructure} but with the additional properties that $\mathcal{P}$ is balanced and $S(\mathcal{P})$ has few connected components. 
	In this section, we state a result (Lemma~\ref{lma:balancingpathsystemcombined}) that allows us to adjust a partition $\mathcal{P}$ to have these additional properties.
	In particular, if $\mathcal{P}$ is not balanced, Lemma~\ref{lma:balancingpathsystemcombined} guarantees us a so-called $\mathcal{P}$-balanced path system in $G$ whose ``contraction''  makes $\mathcal{P}$ balanced in a suitable way. At the end of the section we show how Lemma~\ref{lma:balancingpathsystemcombined} implies our main result, Theorem~\ref{thm:main}.
	
	Let $\mathcal{P} =\{V_{ij}:i,j \in [k]\}$ be a partition of a vertex set~$V$ and 
	let $G$ be a digraph on~$V$. 
	Recall the definition of~$G_{ij}$ and of $P$ being balanced at the beginning of Section~\ref{sec:longcycles}.
	Write $G_{i*} = \bigcup_{j\neq i} G_{ij}$ and $G_{*j} = \bigcup_{i \ne j} G_{ij}$.
	Let $\mathcal{B}(G, \mathcal{P}) = G - \bigcup_{i \in [k]} G_{ii}$.
	Note that $\mathcal{B}(G, \mathcal{P}) = \bigcup_{i \in [k]} G_{i*} = \bigcup_{j \in [k]} G_{*j}$.
	We say that a digraph~$H$ on~$V$ (where $H$ is usually a path system in $G$) is \emph{$\mathcal{P}$-balanced} if, for all~$i \in [k]$, 
	\begin{align*}
		|V_{i*}|-|V_{*i}| =  e_H(V_{i*},V) - e_H(V,V_{*i}) = \sum_{j \in [k]} e(H_{ij}) - \sum_{j \in [k]} e(H_{ji}) = e(H_{i*}) - e(H_{*i}).
	\end{align*}
	The above gives three equivalent conditions for $H$ to be $\mathcal{P}$-balanced. Note that if $H$ is regular, then it is $\mathcal{P}$-balanced for any $\mathcal{P}$.
	
	Let $Q$ be a path in $G$ from~$v_+ \in V_{i_+j_+}$ to~$v_- \in V_{i_-j_-}$.
	We define the \emph{$Q$-contracted subgraph~$G'$ of~$G$} to be $G \setminus V(Q)$ together with a new vertex $w$ such that $N^+_{G'}(w) = N^+_G(v_-) \setminus V(Q)$ and $N^-_{G'}(w) = N^-_G(v_+) \setminus V(Q)$.
	We call $\mathcal{P}'=\{V_{ij}':i,j \in [k]\}$ the \emph{$Q$-contracted partition of $\mathcal{P}$}, where
	\begin{align}
		V_{ij}' = \begin{cases}
			( V_{ij} \setminus V(Q) )  \cup \{w\} & \text{if $(i,j) = (i_-,j_+)$};\\
			V_{ij} \setminus V(Q) &\text{otherwise}.
		\end{cases}
	\end{align}
	Let $\mathcal{Q}$ be a path system in $G$.
	The \emph{$\mathcal{Q}$-contracted subgraph of~$G$} (and \emph{$\mathcal{Q}$-contracted partition of~$\mathcal{P}$}) is obtained by successively contracting each $Q \in \mathcal{Q}$ for~$G$ (and~$\mathcal{P}$, respectively). 
	
	The following two propositions follow from the definition of $\mathcal{Q}$-contraction.
	In fact the definitions are chosen precisely so that these propositions hold.
	
	\begin{proposition} \label{prop:1-fact}
		Let $G$ be a digraph and $\mathcal{Q}$ be a path system in~$G$.
		Suppose that the $\mathcal{Q}$-contracted subgraph of~$G$ contains a $1$-factor with $\ell$ cycles. 
		Then $G$ also contains a $1$-factor containing~$\mathcal{Q}$ with $\ell$ cycles. (The new $1$-factor is simply the old $1$-factor with the paths in $\mathcal{Q}$ uncontracted.)
	\end{proposition}

	\begin{proposition}\label{prop:min-degree-contraction2}
		Let $G$ be a digraph on $n$ vertices and let $\mathcal{P}=\{V_{ij}:i,j\in [k]\}$ be a partition of~$V(G)$.
		Let $\mathcal{Q}$ be a path system in~$G$ such that $e(\mathcal{Q}) < |V_{i*}|, |V_{*i}|$ for all $i \in [k]$. 
		Let $G'$ and $\mathcal{P}' = \{ V'_{ij} : i,j \in [k]\}$ be the $\mathcal{Q}$-contracted subgraph of~$G$  and $\mathcal{Q}$-contracted partition of~$\mathcal{P}$, respectively. 
		Then for all $i \in [k]$,  $\delta(G'[V'_{i*},V'_{*i}] )\geq \delta(G[V_{i*},V_{*i}] ) - 2 e(\mathcal{Q})$.\footnote{Note that, in the definition of contraction, the new vertices created  are placed in exactly the right parts to ensure the degree condition in the proposition.}
	\end{proposition}

	\begin{proposition}\label{prop:contracting}
		Let $\mathcal{P}=\{V_{ij}:i,j\in [k]\}$ be a partition of a vertex set~$V$ and $\mathcal{Q}$ be a $\mathcal{P}$-balanced path system on~$V$.
		Then the $\mathcal{Q}$-contracted partition of $\mathcal{P}$ is balanced. 
	\end{proposition}
	
	\begin{proof}
		Let $Q \in \mathcal{Q}$ be a path from~$v_+ \in V_{i_+j_+}$ to~$v_- \in V_{i_-j_-}$.
		For each $i,j \in [k]$, let $q_{ij} = e(Q_{ij})$, i.e.\ the number of edges of $Q$ from $V_{i*}$ to $V_{*j}$.
		Let $\mathcal{P}' = \{ V_{ij}' : i,j \in [k] \} $ be the $Q$-contracted partition of~$\mathcal{P}$.
		Consider $i \in [k]$. 
		Note that
		\begin{align*}
			|V'_{i*}| &=  |V_{i*}| - |V_{i*} \cap V(Q) | + \indicate({i = i_-}) = |V_{i*}| - \sum_{j \in [k]} q_{ij}, \\
			|V'_{*i}| &=  |V_{*i}| - |V_{*i} \cap V(Q) | + \indicate({i = j_+}) = |V_{*i}| - \sum_{j \in [k]} q_{ji}.
		\end{align*}
		Therefore, $|V'_{i*}| - |V'_{*i}| = |V_{i*}| - |V_{*i}| - \sum_{j \in [k]} (q_{ij} - q_{ji}) = |V_{i*}| - |V_{*i}| - ( e(Q_{i*}) - e(Q_{*i}) )$. 
		A similar statement holds for the $\mathcal{Q}$-contracted partition~$\mathcal{P}^* = \{ V_{ij}^* : i,j \in [k] \}$ of~$\mathcal{P}$, and
			the result follows as $\mathcal{Q}$ is $\mathcal{P}$-balanced (i.e.\ for each $i \in [k]$, $ |V^*_{i*}| - |V^*_{*i}|=|V_{i*}| - |V_{*i}| - ( e(\mathcal{Q}_{i*}) - e(\mathcal{Q}_{*i}) )=0$).
	\end{proof}

	The next lemma says that one can adjust the vertex partition~$\mathcal{P}$  found in Theorem~\ref{thm:digraphstructure}  by a small amount, and find a $\mathcal{P}$-balanced path system~$\mathcal{Q}$ such that $S(\mathcal{P}^*)$ has few connected components, where $\mathcal{P}^*$ is the $\mathcal{Q}$-contracted partition of~$\mathcal{P}$. It is exactly what we need to prove Theorem~\ref{thm:main}, as we shall see.
	In Section~\ref{sec:pfbalancingpathsystem}, we break this lemma down into two further lemmas.
	
	\begin{lemma} \label{lma:balancingpathsystemcombined}
		Let $1/n \ll \gamma \ll \alpha, 1/k$.
		Let $G$ be a $d$-regular digraph on $n$ vertices with $d \ge \alpha n$ and $\mathcal{P}=\{V_{ij}:i,j\in [k]\}$ be a partition of~$V(G)$ such that, for all $i,j \in [k]$, 
		\begin{enumerate}[label ={\rm(\roman*)}]
			\label{itm:balancingpathsystemcombined:1}
			\item  $\delta( G [ V_{i*},V_{*i} ] ) \ge d/k$; 
			\item  $\left| |V_{i*}| - |V_{*i}| \right| \le \gamma n$; 
			\item  $ |V_{i*}|, |V_{*i}|  \ge d - \gamma n $;\label{itm:balancingpathsystemcombined:5}
   
      \item all but at most $\gamma n $ vertices $x \in V_{ij}$ satisfy $d^+(x, V_{*i}), d^-(x, V_{j*})  \ge d - \gamma n $.\label{itm:NEW-DEGREE-CONDITION}
		\end{enumerate}
		Then there exists a partition $\mathcal{P}' = \{ V'_{ij} : i,j \in [k]\}$ of~$V(G)$ and a $\mathcal{P}'$-balanced path system~$\mathcal{Q}$ in~$G$ such that 
		\begin{enumerate}[label ={\rm(\alph*)}]
			\item for all $i, j \in [k]$, $| V_{ij} \triangle V'_{ij}| \le \gamma^{1/2} n$;
			\label{itm:balancingpathsystemcombined:a}
			\item for all $i \in [k], \delta(G[V_{i*}', V_{*i}']) \geq (d/k) - \gamma^{1/2}n$;
			\label{itm:balancingpathsystemcombined:b}    
			\item $ e(\mathcal{Q}) \le \gamma^{1/2} n$;
			\label{itm:balancingpathsystemcombined:c}
			\item $S(\mathcal{P}^*)$ has at most $n/(qd+1)$ connected components, where $\mathcal{P}^*$ is the $\mathcal{Q}$-contracted partition of~$\mathcal{P}'$, and we set $q = 2$ if $G$ is an oriented graph and $q=1$ otherwise.
			\label{itm:balancingpathsystemcombined:d}
		\end{enumerate}
	\end{lemma}

	\subsection{Proof of Theorem~\ref{thm:main}}
	
	We now prove Theorem~\ref{thm:main} assuming Lemma~\ref{lma:balancingpathsystemcombined}.
	
	\begin{proof}[Proof of Theorem~\ref{thm:main}]
		Choose a non-decreasing function $g : (0, 1) \rightarrow (0, 1)$  such that the requirements of Lemmas~\ref{lma:longcycle2} (with $m = \alpha n/2$) and~\ref{lma:balancingpathsystemcombined} are satisfied whenever $n, \gamma, \nu, \tau, \alpha$ satisfy
		\begin{align*}
			1/n \ll_g \gamma \ll_g \nu \leq \tau \ll_g \alpha,
		\end{align*}
		where we write $a \ll_g b$ to mean $a \leq g(b)$.
		Set $f(x) := \min(x, g(x^2/4))$. Fix $\tau \leq f(\alpha)$ and apply Theorem~\ref{thm:digraphstructure} to obtain $n_0$ such that for any $d$-regular digraph $G$ on $n \geq n_0$ vertices with $d \ge \alpha n $, there exists $\gamma$ and $\nu$ such that 
		\begin{align*}
			1/n \ll_f \gamma \ll_f \nu \leq \tau \ll_f \alpha
		\end{align*}
		and a partition~$\mathcal{P} = \{ V_{ij}: i,j \in [k]\}$ of~$V(G)$ satisfying, for all $i,j \in [k]$,
		\begin{enumerate}[label ={\rm(a$_{\arabic*}$)}]
			\item $G[V_{i*},V_{*i}]$ is a bipartite robust $(\nu, \tau)$-expander with partition~$V_{i*},V_{*i}$; \label{item:bip-rob-exp}
			\item for all $x \in V_{ij}$ and $i',j' \in [k]$, $d^+(x, V_{*i}) \ge d^+(x, V_{*i'})$ and $d^-(x, V_{j*}) \ge d^-(x, V_{j'*})$, so in particular, $\delta( G [ V_{i*},V_{*i} ] ) \ge d/k$; \label{item:large-good-degree}
			\item $\left| |V_{i*}| - |V_{*i}| \right| \le \gamma n$; \label{item:almost-balanced}
			\item all but at most $\gamma n $ vertices $x \in V_{ij}$ satisfy $d^+(x, V_{*i}), d^-(x, V_{j*})  \ge d - \gamma n $; \label{item:small-bad-edges}
			\item $k \le   n / (d-\gamma n)$. \label{item:number-of-parts}
		\end{enumerate}
		Note that~\ref{item:number-of-parts} in particular implies that $\gamma \ll 1/k$. Furthermore, \ref{item:large-good-degree} and \ref{item:small-bad-edges} together with $\gamma \ll 1/k$ imply that 
		\begin{enumerate}[label ={\rm(a$_{\arabic*}$)},resume]
			\item for all $i \in [k]$, $|V_{i*}|,|V_{*i}|\ge d-\gamma n$. 
			\label{itm:each-part-is-at-least-d}   
		\end{enumerate}
		
		Let $q=2$ if $G$ is oriented, and $q=1$ otherwise. 
		Apply Lemma~\ref{lma:balancingpathsystemcombined} and obtain a partition $\mathcal{P}' = \{ V'_{ij} : i,j \in [k]\}$ of~$V(G)$ and a $\mathcal{P}'$-balanced path system~$\mathcal{Q}$ in~$G$ such that 
		\begin{enumerate}[label ={\rm(b$_{\arabic*}$)}]
			\item for all $i, j \in [k]$, $| V_{ij} \triangle V'_{ij}| \le \gamma^{1/2} n$; \label{itm:b_1}
			\item for all $i \in [k], \delta(G[V_{i*}', V_{*i}']) \geq (d/k) - \gamma^{1/2}n$;
			\label{itm:b_deg}
			\item $ e(\mathcal{Q}) \le \gamma^{1/2} n$; \label{itm:b_2}
			\item $S(\mathcal{P}^*)$ has at most $n/(qd+1)$ connected components, where $\mathcal{P}^*$ is the $\mathcal{Q}$-contracted partition of~$\mathcal{P}'$. \label{itm:b_3}
		\end{enumerate}
		
		Let $\mathcal{P}^* = \{V^*_{ij}:i,j \in [k]\}$ and let $G^*$ be the $\mathcal{Q}$-contracted subgraph of~$G$. 
		We will check that, for all $i, j \in [k]$,
		\begin{enumerate}[label ={\rm(c$_{\arabic*}$)}]
			\item $\mathcal{P}^*$ is balanced; \label{itm:c_1}
			\item $| V_{ij} \triangle V^*_{ij}| \le 5 \gamma^{1/2} n$; \label{itm:c_2}
			\item $|V_{i*}^*|,|V_{*i}^*|\ge d- 6 k\gamma^{1/2} n \ge d/2$; \label{itm:c_3}
			\item  $G^*[V^*_{i*},V^*_{*i}]$ is a bipartite robust $(\nu/2, 2\tau)$-expander with bipartition $V^*_{i*},V^*_{*i}$; \label{itm:c_4}
			\item  $\delta( G^*[V^*_{i*},V^*_{*i}] ) \ge \alpha^2 |V^*_{i*}|/2$. \label{itm:c_5}
		\end{enumerate}
		Note that \ref{itm:c_1} holds by Proposition~\ref{prop:contracting}. 
		Consider $i,j \in [k]$.
		Note that 
		\begin{align*}
			|V^*_{ij} \triangle V_{ij}| 
			\le |V^*_{ij} \triangle V'_{ij}| + |V'_{ij} \triangle V_{ij}|
			\le 2 |V (\mathcal{Q}) |+|V'_{ij} \triangle V_{ij}|
			\overset{\mathclap{\text{\ref{itm:b_1},\ref{itm:b_2}}}}{\le} 4\gamma^{1/2} n +\gamma^{1/2} n = 5 \gamma^{1/2}n
		\end{align*}
		implying~\ref{itm:c_2}.
		Hence \ref{itm:c_3} follows from \ref{itm:each-part-is-at-least-d} and~\ref{itm:c_2} as $\gamma \ll \alpha,1/k$. By~\ref{item:bip-rob-exp}, \ref{itm:c_2}, \ref{itm:c_3}, 
		and $\gamma \ll \nu$,  Lemma~\ref{lma:birobexpdifferent} (applied to the bipartite graph $G^*[V^*_{i*},V^*_{*i}]\cup G[V_{i*},V_{*i}]$ with $V_{i*}, V_{*i}, V^*_{i*}, V^*_{*i}$ playing the roles of $A,B, A', B'$ respectively) implies that \ref{itm:c_4} holds. By
		\ref{itm:each-part-is-at-least-d}, \ref{itm:b_1}, and \ref{itm:b_2}, we have $e(\mathcal{Q})\leq \gamma^{1/2}n<d-\gamma n-k\gamma^{1/2}n\leq |V'_{i*}|,|V'_{*i}|$ as $\gamma\ll \alpha,1/k$. Hence, by Proposition~\ref{prop:min-degree-contraction2},
		\begin{align*}
			\delta( G^*[V^*_{i*},V^*_{*i}] ) 
			& \ge \delta( G [ V'_{i*},V'_{*i} ] ) - 2e(\mathcal{Q})
			\overset{\mathclap{\text{\ref{itm:b_deg},\ref{itm:b_2}}}}{\ge}  
			(d/k) - 3\gamma^{1/2}n  \geq  \alpha^2 |V^*_{i*}|/2,
		\end{align*}
		where the last inequality holds as $\gamma \ll \alpha, 1/k$ and  $k (\alpha - \gamma)\leq1$ by~\ref{item:number-of-parts}.
		Hence \ref{itm:c_5} holds.
		
		Apply Lemma~\ref{lma:longcycle2} (with $G^*$, $\mathcal{P}^*$, $d/2$, $\nu/2$, $2\tau$, $\alpha^2/2$ playing the roles of $G$, $\mathcal{P}$, $m$, $\nu$, $\tau$, $\alpha$, respectively) to each connected component $I$ in $S(\mathcal{P}^*)$.
		Thus, by~\ref{itm:b_3}, $G^*$ can be covered with at most $n/(qd + 1)$ vertex-disjoint cycles (where we have a cycle on $\cup_{i \in I} (V_{i*} \cup V_{*i})$ for each component $I$ of $S(\mathcal{P}^*)$ so that each cycle has length at least $d/2$ by \ref{itm:c_3}).
		By Proposition~\ref{prop:1-fact}, $G$ can be covered with at most $n/(qd+1)$ vertex-disjoint cycles (each of length at least $d/2$). 
	\end{proof}

	
	\section{Proof of Lemma~\ref{lma:balancingpathsystemcombined}} 
	\label{sec:pfbalancingpathsystem}

	In the last section we reduced the task of proving Theorem~\ref{thm:main} to the task of proving Lemma~\ref{lma:balancingpathsystemcombined}.
	In this section we reduce this further to the task of proving~Lemmas~\ref{lem:balancedpath} and~\ref{lma:regdigraph}. 
	We start by introducing the notion of a \emph{non-trivial path system}, which will help us control the number of connected components of $S(\mathcal{P}^*)$ (as required in the conclusion of Lemma~\ref{lma:balancingpathsystemcombined}).
	
	Recall the definitions at the beginning of Sections \ref{sec:longcycles} and \ref{sec:Path-Contraction}.
	Let $\mathcal{P} =\{V_{ij}:i,j \in [k]\}$ be a partition of a vertex set~$V$. 
	We call a path system~$\mathcal{Q}$ on~$V$  \emph{non-trivial} if there exists $i \ne j$ such that $V_{ij} \setminus V(\mathcal{Q}) \ne \emptyset$ or $\mathcal{Q}$ contains a path from~$V_{*j}$ to~$V_{i*}$. 
	Otherwise, we call $\mathcal{Q}$ a \emph{trivial} path system.
	Note that if $\mathcal{Q}$ is non-trivial, then the $\mathcal{Q}$-contracted partition~$\mathcal{P}'$ of~$\mathcal{P}$ has the property that $S(\mathcal{P}')$ has at least one edge so that its number of connected components is strictly less than its number of vertices.

	Lemma~\ref{lem:balancedpath} below shows that, under the same hypothesis as Lemma~\ref{lma:balancingpathsystemcombined},  there exists a $\mathcal{P}$-balanced path system with few edges, and moreover that this path system is non-trivial if further assumptions are made. We shall see at the end of the section that the existence of such a non-trivial $\mathcal{P}$-balanced path system is enough to prove Lemma~\ref{lma:balancingpathsystemcombined}. 
	In fact, a $\mathcal{P}$-balanced path system alone (without the condition of being non-trivial) is enough to prove Lemma~\ref{lma:balancingpathsystemcombined} except in the extremal cases when $d \approx n/k$ if $G$ is digraph and when $d \approx n/ 2k$ if $G$ is an oriented graph. 
	
	\begin{lemma} \label{lem:balancedpath}
		Let $1/n \ll \gamma \ll 1/k, \alpha$. 
		Let $G^0$ be a $d$-regular digraph on $n$ vertices with $d \ge \alpha n$ and $\mathcal{P}=\{V_{ij}:i,j\in [k]\}$ be a partition of~$V(G^0)$ such that, for all $i \in [k]$,
		\begin{enumerate}[label ={\rm(\roman*)}]
			\item  $\delta( G^0 [ V_{i*},V_{*i} ] ) \ge d/k$; \label{item:good-degree-condition}
			\item  $\left| |V_{i*}| - |V_{*i}| \right| \le \gamma n$; \label{itm:balancedpath3}
			\item  $ |V_{i*}|, |V_{*i}|  \ge d - \gamma n $.
		\end{enumerate}
		Then $G^0$ contains a $\mathcal{P}$-balanced path system~$\mathcal{Q}$ such that $e(\mathcal{Q}) \le k^2  \gamma n $. 
		Moreover, if one of the following holds
		\begin{enumerate}[label={\rm (m\arabic*)}]
			\item $\sum_{i,j \in [k] \colon i \ne j} |V_{ij}| > k^2\gamma n$, or \label{item:size-condition-V-ij} 
			\item there exists $v_0 \in V_{i_0j_0}$ for some $i_0 \ne j_0$ such that 
			$d^{+}(v_0,V_{*i_0})-d^{-}(v_0,V_{i_0*})\ge 100 k^{12} \gamma^{1/3}d$, \label{item:degree-condition-v0} 
		\end{enumerate}
		then we may assume that $\mathcal{Q}$ is non-trivial.
	\end{lemma}

	In case Lemma~\ref{lma:balancingpathsystemcombined} gives a $\mathcal{P}$-balanced but trivial path system (i.e. when both \ref{item:size-condition-V-ij} and \ref{item:degree-condition-v0} fail), 
	we use the next proposition to slightly modify $\mathcal{P}$ and Lemma~\ref{lma:regdigraph} to find the desired non-trivial, $\mathcal{P}$-balanced path system.

	\begin{proposition} \label{prop:balancepartition}
		Let $1/n \ll \gamma \ll 1/k, \alpha$. 
		Let $G^0$ be a $d$-regular digraph on $n$ vertices with $d \ge \alpha n$ and $\mathcal{P}=\{V_{ij}:i,j\in [k]\}$ be a partition of~$V(G^0)$ such that, for all distinct $i,j \in [k]$, 
		\begin{enumerate}[label ={\rm(\roman*)}]
			\item  $\delta( G^0 [ V_{i*},V_{*i} ] ) \ge d/k$; \label{itm:balance1}
			\item $\sum_{i',j' \in [k] \colon i' \ne j'} |V_{i'j'}| \le \gamma n$; \label{itm:balance4}
			\item for all $ v \in V_{ij}$, we have $ d^+(v,V_{*i}) - d^-(v, V_{i *}) \le \gamma n $.  \label{itm:balance5}
		\end{enumerate}
		Then there exists a partition $\mathcal{P}' = \{V_{ii}':i \in [k]\}$ of~$V(G^0)$ such that, for all $i \in [k]$, 
		\begin{enumerate}[label ={\rm(\roman*$'$)}]
			\item  $|V_{ii} \triangle V'_{ii}| \le \gamma n $; \label{itm:balance'2}
			\item  $\delta^0( G^0 [ V_{ii}' ] ) \ge d/k - \gamma n $. \label{itm:balance'1}
		\end{enumerate}
	\end{proposition}
	
	\begin{proof}
		For each $i \in [k]$, let  $V_{ii}' =V_{i*}$. 
		Clearly, $\mathcal{P}' = \{V_{ii}':i \in [k]\}$ is a partition of~$V(G^0)$. 
		Note that \ref{itm:balance4} implies~\ref{itm:balance'2}.
		For $v \in V_{ij}$ (possibly $i=j$), note that
		\begin{align*}
			d^{+}(v, V_{ii}')= d^{+}(v, V_{i*})\ge d^{+}(v, V_{ii}) \overset{\mathclap{\text{\ref{itm:balance4}}}}{\ge} d^{+}(v, V_{*i})-\gamma n \overset{\mathclap{\text{\ref{itm:balance1}}}}{\ge} d/k-\gamma n.   
		\end{align*}
		If $i\neq j$, similarly, we have 
		\begin{align*}
			d^{-}(v, V_{ii}')= d^{-}(v, V_{i*})\overset{\mathclap{\text{\ref{itm:balance5}}}}{\ge}d^{+}(v, V_{*i})-\gamma n \overset{\mathclap{\text{\ref{itm:balance1}}}}{\ge}d/k-\gamma n.
		\end{align*}
		Also, for all $v\in V_{ii}$, \ref{itm:balance1} implies $d^{-}(v, V_{ii}')= d^{-}(v, V_{i*})\ge d/k$. Hence~\ref{itm:balance'1} holds.
	\end{proof}

	\begin{lemma} \label{lma:regdigraph} 
		Let $d,k \in \mathbb{N}$ be such that $k\geq2$ and $d > 165k^5$.
		Let $G$ be a $d$-regular digraph on $n$ vertices and $\mathcal{P}=\{V_{ii}: i \in [k]\}$ be a partition of~$V(G)$ such that, for all $i \in [k]$,
		\begin{enumerate}[label ={\rm(\roman*)}]
			\item $|V_{ii}|\ge d/2$; \label{item:each-part-is-at-least-d-over-2}
			\item $n<(2d+1)k$, and $n<(d+1)k$ if $G$ is not oriented;\label{item:relation-n-d-k}
			\item for all $v \in V_{ii}$, $d^+(v,V_{ii})+d^-(v,V_{ii})\ge d/k$.\label{item:degree-is-large-inside-diagonal}
		\end{enumerate}
		Then $G$ contains a non-trivial $\mathcal{P}$-balanced path system~$\mathcal{Q}$ with $e (\mathcal{Q}) \le k$.
	\end{lemma}

	Before we can prove Lemma~\ref{lma:balancingpathsystemcombined}, we need a good lower bound in oriented graphs for $|\bigcup_{i \in I} \left( V_{i*} \cup V_{*i} \right)|$ for any connected component~$I$ in~$S(\mathcal{P})$; see Proposition~\ref{prop:orientedcomponent}.
	For this, we  use the following result on the minimum semi-degree threshold for an oriented graph to be a robust outexpander.
	Recall that the definition of robust outexpander is given immediately after Remark~\ref{rem:remark1}.
	
	\begin{lemma}[{\cite[Lemma~13.1]{On-Kelly-Conjecture}}]\label{lem:oriented-3-over-8-robust-expander}
		Let $1/n \ll \nu \ll \tau \leq \varepsilon/2 \leq 1$.
		Then every oriented graph~$G$ on $n$ vertices with $\delta^0(G)\ge (3/8+\varepsilon)n$ is a robust $(\nu,\tau)$-outexpander. 
	\end{lemma}

	\begin{proposition} \label{prop:orientedcomponent}
		Let $1/ n \ll \gamma \ll \alpha, 1/k$. 
		Let $G$ be an  oriented graph on~$n$ vertices with $\Delta^0(G)\le d$ where $ d \ge \alpha n$.
		Let $\mathcal{P}=\{V_{ij}:i,j\in [k]\}$ be a partition of~$V(G)$ such that $|V_{i*}| \ge  d/2$ for all $i \in [k]$. 
		Suppose that, for all $i,j \in [k]$, all but at most $\gamma n$ vertices $x\in V_{ij}$ satisfy $d^+(x, V_{*i}), d^-(x, V_{j*})  \ge d - \gamma n $.
		Let $I \subseteq S(\mathcal{P})$ be a connected component. 
		Then
		\begin{align*} 
			\left| \bigcup_{i \in I } (V_{i*} \cup V_{*i}) \right| \ge 
			\begin{cases} 
				2(d -  \gamma n) & \text{if $|I| =1$,}\\
				9(d - 5\gamma n)/4 & \text{otherwise.}
			\end{cases}
		\end{align*}	
	\end{proposition}
	
	\begin{proof}
		Without loss of generality, let $I = [\ell]$. 
		For each $i \in [\ell]$, we can find vertices $ v_i \in V_{i*}$ such that $d^+(v_i, V_{*i})\ge d - \gamma n $ as $\gamma \ll \alpha,1/k$, which shows in particular that $|V_{*i}|\geq d-\gamma n$.
		Hence, if $\ell \ge 3$,
		\begin{align*}
			\left|\bigcup_{i \in I } (V_{i*} \cup V_{*i})\right|
			\ge  \sum_{i \in [\ell] }  | V_{*i}|  \ge \ell (d - \gamma n)  \ge 3(d-\gamma n)\ge 9(d-5\gamma n)/4.
		\end{align*}

		If $\ell = 1$, then $V_{11} = V_{1*} = V_{*1}$. 
		There exists a vertex $v \in V_{11}$ such that  $d^+(v, V_{11}),  d^-(v, V_{11}) \ge d - \gamma n $. 
		Since $G$ is an oriented graph, we have $|V_{11}| \ge d^+(v, V_{11}) +  d^-(v, V_{11}) \ge 2(d - \gamma n)$. 
		
		If $\ell = 2$, then $\bigcup_{i \in [2] } V_{i*} = \bigcup_{i \in [2] } V_{*i}$ and $V_{ij} = V_{ji} = \emptyset$ for all $(i,j) \in [2] \times ([k] \setminus [2])$. 
		Let $V_I = \bigcup_{i \in [2] } V_{i*}=\bigcup_{i \in [2] } V_{*i}$.
		There exists a vertex subset $W  \subseteq V_I$ such that
		\begin{align*}
			|W| \ge |V_I| - 4 \gamma n \ge 2(d-\gamma n)-4\gamma n \ge \alpha n,
		\end{align*}
		and such that, writing $W_{ij} = W \cap V_{ij}$ for $i,j \in [2]$, we have for all $i,j \in [2]$ that
		\begin{align}
			\label{eq:degW}
			d^+(x, W_{*i}),   d^-(y, W_{j*}) \ge d - 5\gamma n  \text{ for all } x \in W_{i*} \text{ and } y \in W_{*j}.
		\end{align}
		Hence $\delta^0( G[W] ) \ge d - 5\gamma n$.

		Let $\tau$ be a constant with $\gamma \ll \tau \ll \alpha$. Next we show that $G[W]$ is not a robust $( \gamma^{1/3} , \tau)$-outexpander.
		For $i \in [2]$, \eqref{eq:degW}  implies that $|W_{i*}| \geq d - 5 \gamma n  \ge \tau n \ge \tau |W|$ and so $\tau |W| \leq |W_{i*}| \leq (1 - \tau)|W|$. 
		Since $W_{1*} \cup W_{2*}  = W_{*1} \cup W_{*2}$, we may assume without loss of generality that $|W_{1*}| \geq |W_{*1}|$.
		Recall that $\Delta^0(G) \le d$. 
		By~\eqref{eq:degW}, each vertex in $|W_{1*}|$ has at most $5\gamma n$ outneighbours in~$W_{*2}$, so 
		\begin{align*}
			|{\rm RN}_{\gamma^{1/3}}^+(W_{1*}) \cap W_{*2}| 
			\leq \frac{e_G(W_{1*}, W_{*2})}{ \gamma^{1/3} |W|} 
			\le \frac{5 \gamma n^2} {\alpha \gamma^{1/3} n } 
			< \gamma^{1/3} \alpha n 
			\le \gamma^{1/3} |W|,
		\end{align*}
		where the penultimate inequality holds as $\gamma \ll \alpha$. 
		This implies that 
		\begin{align*}
			|{\rm RN}_{\gamma^{1/3}}^+(W_{1*})|
			\le |W_{*1}| + |{\rm RN}_{\gamma^{1/3}}^+(W_{1*}) \cap W_{*2}|
			< |W_{1*}| + \gamma^{1/3} |W|.
		\end{align*}
		Hence $G[W]$ is not a robust $( \gamma^{1/3}, \tau)$-outexpander as claimed. 
		Lemma~\ref{lem:oriented-3-over-8-robust-expander} with $\varepsilon=5/72$ implies 
		$\delta^0( G[W] ) \leq (3/8 + \varepsilon )|W| = 4|W|/9$. 
		Then $4|W|/9 \geq d - 5\gamma n$, and the result follows.
	\end{proof}

	We now prove Lemma~\ref{lma:balancingpathsystemcombined} assuming Lemmas~\ref{lem:balancedpath} and~\ref{lma:regdigraph}.
	
	\begin{proof}[Proof of Lemma~\ref{lma:balancingpathsystemcombined}]
		Let $G$ and $\mathcal{P} = \{V_{ij}: i,j \in [k]\}$ be as in the statement of Lemma~\ref{lma:balancingpathsystemcombined}. 
		We apply Lemma~\ref{lem:balancedpath} (with $G^0 = G$) and obtain a $\mathcal{P}$-balanced path system~$\mathcal{Q}^0$ such that $e(\mathcal{Q}^0) \le k^2  \gamma n $. 
Let $\mathcal{P}^{0}$ be the $\mathcal{Q}^0$-contracted partition of $\mathcal{P}$.
		Let $k^*$ be the smallest integer larger than $n/(qd+1)$, that is, $k^* = \lfloor n/(qd+1) \rfloor +1 $.
		If $S(\mathcal{P}^{0})$ has at most $n/(qd+1)$ connected components, then we are done by setting $\mathcal{P}' = \mathcal{P}$ and $\mathcal{Q}= \mathcal{Q}^0$ since $e(\mathcal{Q}^0) \le k^2  \gamma n\le \gamma^{1/2} n $.
		Hence we may assume that $S(\mathcal{P}^{0})$ has at least $k^*$ connected components.

		\begin{claim}\label{clm:S(P)-empty}
			We have $k^* = k$ and $S(\mathcal{P}^{0})$ is an empty graph on~$[k^*]$. 
		\end{claim}

		\begin{proofclaim}
			Since $1/n\ll \gamma \ll \alpha$ and $d \ge \alpha n$, we have that
			\begin{align}
                \dfrac{n}{q(d-\gamma n)}-\dfrac{n}{qd+1}\leq \dfrac{1}{10}\text{ and so }
				\left\lfloor \frac{n}{ q (d-\gamma n ) } \right\rfloor 
				\le  \left\lfloor \frac{n}{q d+1} \right\rfloor +1 =k^*.
				\label{eqn:k^*}
			\end{align}
			
			First suppose that $G$ is not oriented, so $q=1$. 
			Recall \ref{itm:balancingpathsystemcombined:5} that $|V_{i*}| \ge d - \gamma n $ for all $i \in [k]$, so $n\geq k(d-\gamma n)$. 
			Together with~\eqref{eqn:k^*}, we have $k \le \left\lfloor n/ (d-\gamma n) \right\rfloor \le k^*$.
			On the other hand, since $S(\mathcal{P}^{0})$ has at least $k^*$ connected components, we have $k  = | V( S(\mathcal{P}^{0}) ) | \ge k^*$.
			Thus $k = k^*$. 
			Furthermore, $S(\mathcal{P}^{0})$ is an empty graph on~$[k^*]$ (or else $S(\mathcal{P}^{0})$ would have fewer than $k^*$ connected components). 
			
			Next, suppose that $G$ is oriented, so $q=2$. 
			Let $I_1 , \dots ,  I_{a+b}$ be connected components of~$S(\mathcal{P}^{0})$ such that $|I_j| = 1$ if and only if $j \in [a]$. 
			Note that $ a+b \ge k^*$ and $\{ \bigcup_{i \in I_j } (V_{i*} \cup V_{*i}) : j \in [a+b] \}$ is a partition of~$V(G)$. Recall \ref{itm:NEW-DEGREE-CONDITION} that for all $i,j\in [k]$, all but at most $\gamma n $ vertices $x \in V_{ij}$ satisfy $d^+(x, V_{*i}), d^-(x, V_{j*})  \ge d - \gamma n $. Also, we have that $|V_{i*}|\geq d-\gamma n\geq d/2$ for all $i\in[k]$ as $\gamma \ll \alpha$.
			By Proposition~\ref{prop:orientedcomponent}, we have 
   
   			\begin{align*}
				n = \sum_{j \in [a+b]} \left| \bigcup_{i \in I_j } (V_{i*} \cup V_{*i}) \right| 
				\ge 2(d-\gamma n) a + \frac{9(d-5\gamma n)}4 b 
				=2(d-\gamma n)(a+b)+\dfrac{d-37\gamma n}{4}b.
			\end{align*}
		Hence \eqref{eqn:k^*} implies that 
   \begin{align*}
   \dfrac{1}{10}\geq \dfrac{n}{2(d-\gamma n)}-\dfrac{n}{2d+1}>\dfrac{n}{2(d-\gamma n)}-k^*\geq (a+b-k^*)+\dfrac{b(d-37\gamma n)}{8(d-\gamma n)}>(a+b-k^*)+\dfrac{b}{9},
   \end{align*} where the last inequality holds as $\gamma \ll \alpha$. Recalling that $a+b \geq k^*$, this shows that $a+b=k$ and $b = 0 $, and so $ a = k^*$, proving the claim.
		\end{proofclaim}
		
		Since $S(\mathcal{P}^{0})$ is an empty graph on~$[k^*]$, we deduce that $\mathcal{Q}^0$ is trivial as defined at the start of the section. 
		By the moreover statement of Lemma~\ref{lem:balancedpath}, we have 
  
		\begin{enumerate}[label={\rm ($\overline{\text{m}\arabic*}$)}]
			\item $\sum_{i,j \in [k] \colon i \ne j} |V_{ij}| \le k^2\gamma n \le \gamma^{1/4} n$ and 
			\label{itm:m1(int)}
			\item  for all distinct $i,j \in [k]$ and $ v \in V_{ij}$, we have $ d^+(v,V_{*i}) - d^-(v, V_{i *}) \le 100k^{12} \gamma^{1/3}d  \le \gamma^{1/4} n$.
			\label{itm:m2(int)}
		\end{enumerate}
		By Proposition~\ref{prop:balancepartition} (with $\gamma^{1/4}$ playing the role of $\gamma$), there exists a partition $\mathcal{P}' = \{V_{ii}':i \in [k]\}$ of~$V(G)$ (and we take $V_{ij}' = \emptyset$ for $i \not= j$) such that, for all~$i \in [k]$, 
		\begin{enumerate}[label ={\rm(\roman*$'$)}]
			\item  $|V_{ii} \triangle V'_{ii}| \le \gamma^{1/4} n $;
			\label{itm:symdiff}
			\item  $\delta^0( G [ V_{ii}' ] ) \ge (d/k) - \gamma^{1/4} n $. 
			\label{itm:deg}
		\end{enumerate}
		Finally, apply Lemma~\ref{lma:regdigraph} (with $\mathcal{P}'$ playing the role of $\mathcal{P}$) by noting that property~\ref{item:each-part-is-at-least-d-over-2} of Lemma~\ref{lma:regdigraph}  holds by property~\ref{itm:balancingpathsystemcombined:5} of Lemma~\ref{lma:balancingpathsystemcombined} together with \ref{itm:m1(int)} and \ref{itm:symdiff} above, and properties \ref{item:relation-n-d-k} and \ref{item:degree-is-large-inside-diagonal} of Lemma~\ref{lma:regdigraph} hold by Claim~\ref{clm:S(P)-empty} and \ref{itm:deg} above, respectively.
		Thus, in $G$, we obtain a non-trivial $\mathcal{P}'$-balanced path system~$\mathcal{Q}$ with $ e (\mathcal{Q}) \le k \leq \gamma^{1/2}n$. 
		Also, if $\mathcal{P}^*$ is the $\mathcal{Q}$-contracted partition of~$\mathcal{P}'$, then by Claim~\ref{clm:S(P)-empty}, $S(\mathcal{P}^*)$ has at most $k-1=k^* - 1 \leq n/(qd + 1)$ connected components since $\mathcal{Q}$ is non-trivial. This gives \ref{itm:balancingpathsystemcombined:c} and \ref{itm:balancingpathsystemcombined:d} in Lemma~\ref{lma:balancingpathsystemcombined}, while \ref{itm:m1(int)} and \ref{itm:symdiff} imply \ref{itm:balancingpathsystemcombined:a}, and \ref{itm:deg} implies  \ref{itm:balancingpathsystemcombined:b}.
	\end{proof}

	
	\section{Proof of Lemma~\ref{lem:balancedpath}}
	\label{sec:pf_balancedpath}

	Recall that in Lemma~\ref{lem:balancedpath}, we are given a $d$-regular digraph $G^0$ on $n$ vertices together with a partition $\mathcal{P}=\{V_{ij}:i,j\in[k]\}$ of $V(G^0)$ satisfying certain size and degree conditions, and we must find a $\mathcal{P}$-balanced path system $\mathcal{Q}$ with few edges. Moreover, we require that $\mathcal{Q}$ is non-trivial under certain circumstances. We begin with a sketch of our proof approach for this lemma (ignoring the moreover part for now). Recall the definitions at the start of Sections \ref{sec:longcycles} and \ref{sec:Path-Contraction}.
	
	Let $V^+$ and $V^-$ be two disjoint copies of~$V(G^0)$. 
	Let $V_{1*}^+, \dots, V_{k*}^+$ to be the partition of~$V^+$ corresponding to $V_{1*}, \dots, V_{k*}$ in~$V(G^0)$.
	Similarly, let $V_{*1}^-, \dots, V_{*k}^-$ to be the partition of~$V^-$ corresponding to $V_{*1}, \dots, V_{*k}$ in~$V(G^0)$.
	Let $H$ be the bipartite digraph with bipartition $V^+, V^-$ such that for each $x \in V^+_{i*}$ and $y \in V^-_{*j}$, we have $xy \in E(H)$ if and only if $xy \in E(G^0)$ and $i \ne j$. 
	So there is a bijection between $E(H)$ and $E(\mathcal{B}( G^0, \mathcal{P}))$.

	A naive approach is to find, for each $i, j \in [k]$ with $i \ne j$, a matching~$M_{ij}$ in~$H[V_{i*}^+, V^-_{*j}]$ with $e(M_{ij}) = e(G^0_{ij})/d$.
	If $Q_{ij}$ is the subgraph in $G^0$ corresponding to~$M_{ij}$, 
	then $Q = \bigcup_{i,j \in [k]} Q_{ij}$ is $\mathcal{P}$-balanced by construction since $G^0$ is $d$-regular. 
	It is relatively easy to guarantee $Q$ has few edges by first passing from $G^0$ to a suitable subdigraph~$G$ (see Proposition~\ref{prop:subgraph}). 
	However, there are three problems with this approach: (i) $\Delta^0(Q)$ might be greater than one (meaning that $Q$ is not a path system), (ii) $e(G^0_{ij})/d$ may not be an integer, and (iii) $Q$ may contain cycles, even if $\Delta^0(Q) \leq 1$. 
	
	To overcome (i) and~(ii), we consider a suitable flow problem by converting~$H$ to a network~$\mathcal{F}^* = (F^*,w, s^*, t^*)$ as follows 
	(the formal definition of a network and flow are stated in Section~\ref{subsect:flows}).
	Starting with the graph $H$ (with all edges of~$H$ having capacity 1), we add new vertices $s^*, s_1, \dots, s_k, t^*, t_1, \dots, t_k$, where $s^*$ and~$t^*$ are the source and sink respectively and where the $s_i$ and~$t_j$ are viewed as `local' sources and sinks, respectively. 
	For each $i,j \in [k]$ and each $x^+ \in V^+_{i*}$, $y^- \in V^-_{*j}$, we add the edges $s_i x^+$ and $y^- t_j$, each of capacity~$1$.
	For each $i \in [k]$, we add the edge $s^* s_i$ of capacity $\max \{|V_{i*}| - |V_{*i}|, 0\}$, the edge $t_i t^*$ of capacity $\max \{|V_{*i}| - |V_{i*}|, 0\}$ and the edge $t_i s_i$ of infinite capacity.
	This gives the network $\mathcal{F}^*$; see Figure~\ref{fig:F*}.
	Consider a maximum integer flow~$f$ for~$\mathcal{F}^*$. 
	Define $Q$ to be the subdigraph of~$\mathcal{B}(G^0, \mathcal{P})$ such that $xy \in E(Q)$ if and only if $x^+ y^-$ has a flow of one in~$f$, where $x^+$ and $y^-$ are the corresponding copies of $x$ and~$y$ in~$V^+$ and~$V^-$, respectively. 
	Note that $\Delta^0(Q) \le 1$ by construction. Also it is easy to check that $Q$ is $\mathcal{P}$-balanced provided all edges at $s^*$ (and hence also at $t^*$) are saturated by~$f$.
	
	To guarantee this latter condition on $f$, it turns out that it is enough to find a fractional flow $f^*$ such that $f_{ij}^*$, the total amount of flow in $f^*$ through the edges in $H[V_{i*},V_{*j}]$, is roughly $e(G^0_{ij})/d$ for all $i\neq j$. Such a fractional flow is almost a maximum flow, and we can use the max-flow min-cut theorem to convert this fractional flow into an integer flow that saturates the edges at $s^*$ (and $t^*$); see~Claim~\ref{clm:fractionalflow}. This also addresses~(ii) above.

	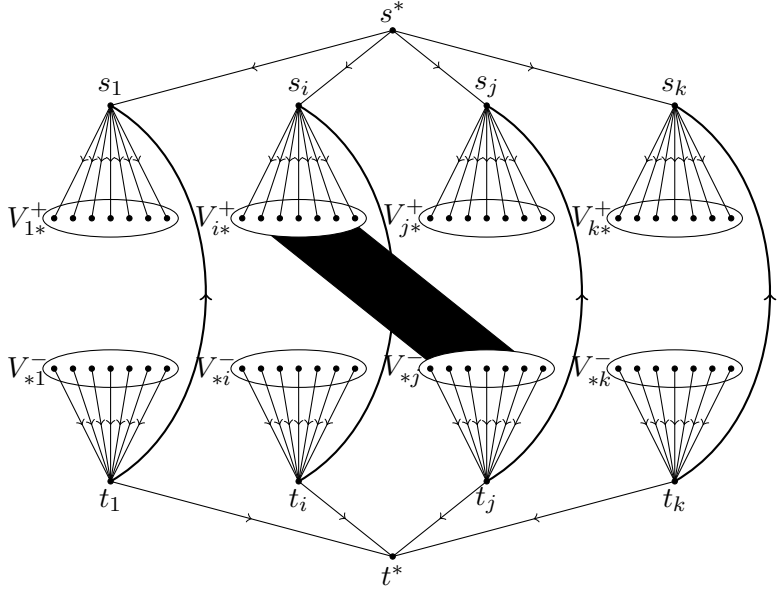
\begin{figure}[!htbp]
		\begin{center}
			\begin{tikzpicture}[scale=0.5,decoration={
					markings, 
					mark=at position 0.5 with {\arrow{>}}}
				]
				\draw[fill = black, opacity=0.3] ( -1.2,2) -- (-3.8,2)--( 1.2,-2) -- (3.8,-2)--(- 1.2,2);
				\node at (0,0) [rotate =-45 ] {$G^0[V_{i*}, V_{*j}]$};
				
				\foreach \i/\j in {-7.5/1,-2.5/i,2.5/j,7.5/k}
				{   
					\begin{scope}    	
						\draw[postaction={decorate}] (0,7) --(\i,5) ;
						\draw[postaction={decorate}] (\i,-5) -- (0,-7) ;
					\end{scope}    
					
					\begin{scope}[xshift= {\i cm} ] 
						
						\draw[fill = white] (0,2) ellipse (1.8 cm and 0.5 cm);
						\draw[fill = white] (0,-2) ellipse (1.8 cm and 0.5 cm);
						\filldraw (0,5) circle (2pt);
						\node at (0,5.5) {$s_{\j}$};
						\filldraw (0,-5) circle (2pt);
						\node at (0,-5.5) {$t_{\j}$};
						\node at (-2.2,2) {$V^+_{\j *}$};
						\node at (-2.2,-2) {$V^-_{* \j }$};
						\foreach \x in { 1.5 , 1.0 ,..., -1.7}
						{
							\filldraw (\x,2) circle (2pt);
							\draw[postaction={decorate}] (0,5) -- (\x,2);
							\filldraw (\x,-2) circle (2pt);
							\draw[postaction={decorate}] (\x,-2) --  (0,-5);
						}
						\draw[thick, postaction={decorate}] (0,-5) to [out = 30,in=-30]  (0,5) ;
					\end{scope}
				}
				
				\filldraw (0,7) circle (2pt);
				\node at (0,7.5) {$s^*$};
				\filldraw (0,-7) circle (2pt);
				\node at (0,-7.5) {$t^*$};

			\end{tikzpicture}
		\end{center}
		\caption{The network $\mathcal{F}^*$}
		\label{fig:F*}
	\end{figure}
	
	To deal with~(iii), suppose we could find a matching~$M^0$ of~$G^0$ such that $e(M^0[V_{i*},V_{*j}]) \ge e(G^0_{ij})/d$ for all $i,j \in [k]$. Then defining~$H$ (and $\mathcal{F}^*$) using~$M^0$ instead of~$G^0$, we could apply the same argument as before to obtain a $\mathcal{P}$-balanced subgraph $Q \subseteq M^0$ (which is therefore necessarily a path system as required). 
	However, it is not always possible to find such a matching, so instead, for each~$i,j$, we find a large ``extendable mathcing'' in $G^0_{ij}$ which consists of a matching $M_{ij} \subseteq G^0_{ij}$ together with sets of vertices $X_{ij}^+ \subseteq V_{i*}$ and $X_{ij}^- \subseteq V_{*j}$ where each vertex $v^+ \in X_{ij}^+$ has high outdegree in $G^0_{ij}$ and each vertex $v^- \in X_{ij}^-$ has high indegree in~$G^0_{ij}$, and where $\bigcup_{i,j} M_{ij}$ is a matching (see Lemma~\ref{lem:matching}). Here the very large degrees of vertices in $X_{ij} = X_{ij}^+ \cup X_{ij}^-$ allow us to greedily extend $M_{ij}$ into a suitably large path system\footnote{When doing the greedy extension we obtain a path system (with paths of length at most 3) rather than a matching because we allow the start (resp.\ end) point of an edge in $M_{ij}$ to be in $X_{ij}^-$ (resp.\ $X_{ij}^+$). } (see Proposition~\ref{prop:extendingmatching}). These path systems would be disjoint (as required) if the $X_{ij}$ are disjoint, but a priori the $X_{ij}$ will not be disjoint. 
	Therefore, we must modify the $X_{ij}$ so that each vertex is assigned to at most one $X_{ij}^+$ and at most one $X_{ij}^-$. It turns out that this assignment problem is not too hard to incorporate into the flow problem described earlier (we simply add an edge of capacity $1$ from $v^+ \in V^+$ to $t_j$ if $v^+ \in X_{ij}^+$ and similarly if $v^- \in X_{ij}^-$). Solving the flow problem gives us the \emph{disjoint} ``extendable matchings'' we seek (meaning the $X_{ij}$'s are disjoint), which can greedily be extended to give the desired path system.

	\subsection{Flows}
	\label{subsect:flows}
	We recall some common definitions and facts about flow networks.
	We note that flows are only used in the proof of Lemma~\ref{lem:balancedpath} (and not in any preliminary results).

	A \emph{flow network} is a tuple $\mathcal{F} = (F,w,S,T)$, where $F$ is a digraph, $w\colon E(F) \to \mathbb{R}_{\geq 0}$ is the \emph{capacity} function, and $S \subset V$ is a set of \emph{sources} (i.e.~each $s \in S$ only has outedges incident to it) and $T \subset V$ is a set of \emph{sinks} (i.e.~each $t \in T$ only has inedges incident to it).
	A \emph{flow} for the flow network~$\mathcal{F}$ is a function $f \colon E(F) \to \mathbb{R}_{\geq 0}$ such that, for all $e\in E$, we have $f(e)\leq w(e)$ and, for all $v\in V\setminus (S \cup T)$, we have $\sum_{u\in N^-_F(v)}f(uv)=\sum_{u\in N^+_F(v)}f(vu)$.
	We define the \emph{value} of $f$ as $val(f)\coloneqq \sum_{s\in S}\sum_{v\in N^+_F(s)} f(sv)=\sum_{t\in T}\sum_{u\in N^-_F(t)} f(ut)$.
	A \emph{maximum flow} in a given flow network is a flow $f$ that maximises~$val(f)$.
	We say that $f$ is an \emph{integer flow} if $f(e)$ is an integer for all $e \in E$, and to emphasize the contrast, we sometimes refer to a flow that is not necessarily an integer flow as a \emph{fractional flow}. For any set of edges $E' \subseteq E$, we write $f(E'):= \sum_{e \in E'}f(e)$.
	If $S = \{s\}$ and $T = \{t\}$, then we simply write $(F,w,s,t)$.
	
	There are variants of the above notions that we will use. 
	In particular, as well as having edge capacities, a flow network can also have vertex capacities (which restricts the amount of flow that can pass through a vertex) so that now $w: E(F) \cup V(F) \to \mathbb{R}^+$, and for each $v \in V(F) \setminus (S \cup T) $ a flow is defined as before with the added restriction that $f(v) \leq w(v)$, where  $f(v):=\sum_{u\in N^-_F(v)}f(uv)=\sum_{u\in N^+_F(v)}f(vu)$. 
	One can easily reduce this to the situation of just edge capacities by replacing each vertex $v$ with a directed edge $v^-v^+$ of capacity $w(v)$ and each directed edge~$uv$ (or~$vw$) with~$uv^-$ (or~$v^+w$, respectively) where new edges inherit the capacities of their old counterparts. 
	We say that a flow~$f$ \emph{saturates} an edge $e$ (or a vertex~$v$) if $f(e) = w(e)$ (or $f(v) = w(v)$, respectively).
	
	Let $(F,w,s,t)$ be a flow network.
	For a partition $(U,W)$ of~$V(F)$ with $s\in U$, $t\in W$  the edge set~$E_F(U,W)$ is called a \emph{cut}.
	(Recall that $E_F(U,W)$ is the set of edges in~$F$ from~$U$ to~$W$.) 
	The \emph{capacity} of a cut $E_F(U,W)$ is the sum of the capacities of its edges, i.e.~$w(E_F(U,W)) \coloneqq \sum_{e \in E_F(U,W)}w(e)$. 
	A \emph{minimum cut} of the given flow network is a cut of minimum capacity. 
	We make use of the following well-known theorem.
	
	\begin{theorem}[Max-flow min-cut~\cite{DantzigFukerson}]
		\label{thm:maxflow}
		Let  $(F,w,s,t)$ be a flow network.
		\begin{enumerate}[label={\rm (\roman*)}]
			\item \label{itm:maxflow1} If $f$ is a flow and $E_F(U,W)$ is a cut then  
			$\mathit{val}(f) \leq \mathit{w}(E_F(U,W))$. 
			\item \label{itm:maxflow2} We have that $f$ is a maximum flow and  $E_F(U,W)$ is a minimum cut if and only if $\mathit{val}(f)=\mathit{w}(E_F(U,W))$ and in that case $f$ saturates every edge in $E_F(U,W)$. 
			\item \label{itm:maxflow3} If all capacities are non-negative integers, i.e.\ $w(e) \in  \mathbb{Z}_{\geq 0}$ for all $e \in E(F)$, then there exists a maximum flow $f$ that is an integer flow.
		\end{enumerate}
	\end{theorem} 
	
	We also use the following whose proof is omitted as it is straightforward.
	
	\begin{proposition}\label{prop:flow-reduction}
		Let $\mathcal{F} = (F,w,s,t)$ be a flow network and $f$ a flow for~$\mathcal{F}$. 
		Then, for any edge~$e\in E(F)$, there exists a flow~$f'$ for~$\mathcal{F}$ such that $f'(e)=0$ and $val(f') \ge val(f)-f(e)$.  
	\end{proposition}

	\subsection{Preliminaries}
	
	We will require Vizing's theorem for multigraphs in the proof of Lemma~\ref{lem:matching}. Let $H$ be an (undirected) multigraph (without loops). The \textit{multiplicity} $\mu(H)$ of $H$ is the maximum number of edges between two vertices of $H$, and, as usual, $\Delta(H)$ is the maximum degree of $H$.
	A \textit{proper $k$-edge-colouring} of $H$ is an assignment of $k$ colours to the edges of $H$ such that incident edges receive different colours.
	
	\begin{theorem}[\cite{Vizing}; see e.g.~\cite{Diestel}]
		\label{thm:Vizing}
		Any multigraph $H$ has a proper $k$-edge colouring with $k = \Delta(H) + \mu(H)$ colours. In particular, by taking the largest colour class, there is a matching in $H$ of size at least $e(H)/(\Delta(G)+\mu(G))$.
	\end{theorem}
	
	We will also require the following simple proposition about decomposing acyclic digraphs into paths.
	\begin{proposition}[{\cite[Proposition 2.6]{LPST}}]
		\label{prop:path-decomposition}
		Let $G$ be an acyclic digraph.
		Then the edges of~$G$ can be partitioned into $\sum_{v \in V(G)}\left| d^+(v) - d^-(v) \right|/2$ directed paths.
	\end{proposition}

	We will need the following property of regular digraphs; its proof is a simple exercise by considering $ \sum_{v \in V_{i*}} d^+(v)$ and $ \sum_{v \in V_{*i}} d^-(v)$.
	
	\begin{proposition}[{\cite[Proposition~3.2]{LPY}}]\label{prop:partition-regular-graphs}
		Let $G$ be a $d$-regular digraph, $k\in \mathbb{N}$, and $\mathcal{P}=\{V_{ij}:i,j\in [k]\}$ be a partition of~$V(G)$. 
		Then, for all $i \in [k]$, we have $d(|V_{i*}|-|V_{*i}|)= e( G_{i*} ) - e(G_{*i})$.
	\end{proposition}

	The next proposition shows that we can consider a simpler subgraph~$G$ in~$\mathcal{B}( G^0 , \mathcal{P} )$.

	\begin{proposition} \label{prop:subgraph}
		Let $G^0$ be a digraph, $k,d \in \mathbb{N}$ and $\mathcal{P}=\{V_{ij}:i,j\in [k]\}$ be a partition of~$V(G^0)$ such that
		for all $i \in [k]$, $e( G^0_{i*} ) - e(G^0_{*i}) = d ( |V_{i*}|-|V_{*i}| )$.
		Then, by reordering $[k]$ if necessary, there exists a subdigraph~$G$ of~$\mathcal{B}(G^0, \mathcal{P})$ such that
		\begin{align*}
			e(G)  & \le  d (k-1) \sum_{i \in [k]}\left| |V_{i*}|-|V_{*i}|  \right|/2,\\
			e( G_{i*} ) - e(G_{*i})   & = d(|V_{i*}|-|V_{*i}|) \text{ for all $i \in [k]$,}
		\end{align*}
		and for all $i,j \in [k]$ with $i \le j$, $e(G_{ji}) = 0$.
	\end{proposition}
	
	\begin{proof}
		Define an auxiliary multidigraph~$H^0$ on~$[k]$ (with loops) such that there are precisely $e(G^0_{ij})$ many $ij$ edges in~$H^0$. 
		There is a natural bijection between $E(G^0)$ and~$E(H^0)$.
		Let $H$ be an acyclic subdigraph of~$H^0$ obtained by successively removing all edges of a directed cycle where we treat loops also as cycles.
		Note that, for all $i \in [k]$,
		\begin{align}
			d^+_{H}(i) - d^-_{H}(i) = 	d^+_{H^0}(i) - d^-_{H^0}(i) = e( G^0_{i*} ) - e(G^0_{*i}) = d(|V_{i*}|-|V_{*i}|). \label{eqn:prop1}
		\end{align}
		Since $H$ is acyclic, by relabelling if necessary, we may assume that there are no edges $ji$ with $i \le j$.
		By Proposition~\ref{prop:path-decomposition}, $E(H)$ can be decomposed into $\sum_{i \in [k]}\left| d^+_{H}(i) - d^-_{H}(i) \right|/2$ directed paths. 
		Since each path can have length at most $k-1$ (as $H$ has $k$ vertices), we have 
		\begin{align*}
			e(H) & 
			\le (k-1) \sum_{i \in [k]}\left| d^+_{H}(i) - d^-_{H}(i) \right| / 2 
			\overset{\mathclap{\text{\eqref{eqn:prop1}}}}{=} d (k-1) \sum_{i \in [k]}\left| |V_{i*}|-|V_{*i}|  \right|/2.
		\end{align*}
		The result follows by setting~$G$ to be the subdigraph of~$G^0$ corresponding to~$H$. 
	\end{proof}
	
	We need the following lemma from our previous work~\cite{LPY}. 
	It states that given a set of matchings of low total maximum degree, one can select a relatively large number of the edges from each matching so that the union of selected edges is also a matching.

	\begin{lemma}[{\cite[Lemma~4.2]{LPY}}]\label{lem:matchinglemma}
		Let $k,\ell\in\mathbb{N}$ and $M_1,M_2,\ldots,M_{\ell}$ be matchings with $\Delta\left(\bigcup_{i \in [\ell]} M_i\right)\leq k$.
		Suppose $e(M_i)> 8k^6\ln \ell$ for all $i \in [\ell]$.
		Then, there exists a matching $M \subseteq \bigcup_{i \in [\ell]}M_i$ with $|M\cap M_i|\geq e(M_i)/2k^2$ for all $i \in [\ell]$. 
	\end{lemma}
	
	Any graph (or digraph) either has a large matching or has all its edges incident to a small set of vertices (so these vertices have relatively large degree).
	The following lemma allows us to interpolate between these extremes and moreover does it simultaneously for all~$G_{ij}$.
	This corresponds to the ``extendable matchings'' described in the sketch of proof.
	
	\begin{lemma} \label{lem:matching}
		Let $G$ be a digraph, $k,d\in \mathbb{N}$, and $\mathcal{P} = \{V_{ij}:i,j\in [k]\}$ be a partition of $V(G)$.
		Let $\theta \in(0,1)$ with $\theta d \ge 2$. 
		For all $i,j \in [k]$ and $v \in V(G)$, let $X_{ij}^+ = \{v \in V(G) : d_{G_{ij}}^+(v) \ge \theta  d \}$ and define $X_{ij}^-$ similarly. 
		Then, for each $i,j \in [k]$, there exists a matching~$M_{ij}$ of~$G_{ij}$ such that
		\begin{enumerate}[label = {\rm(\roman*)}]
			\item $	e(G_{ij})  \le  \sum_{x^+ \in X_{ij}^+}  d_{G_{ij}}^+(x^+) + \sum_{x^- \in X_{ij}^-}  d_{G_{ij}}^-(x^-) + (6 k^2 \theta d) e( M_{ij} ) + 50 k^{10} \theta d  $;
			\item for all $uv \in M_{ij}$, $u \notin X_{ij}^+$ and $v \notin X_{ij}^-$;
			\item $\bigcup_{i,j \in [k]} M_{ij}$ is a matching.
		\end{enumerate}
	\end{lemma}
	
	\begin{proof}
		Consider $i,j \in [k]$. 
		Let $H_{ij}$ be the multigraph obtained from $G_{ij}$ by deleting all the edges~$uv$ with $u\in X_{ij}^+$ or $v \in X_{ij}^-$, and by making all the edges undirected.
		Note that we have $\Delta(H_{ij}) + \mu(H_{ij}) \leq 2\theta d+2$ and
		\begin{align}
			e(H_{ij}) \geq e(G_{ij}) -  \left(  \sum_{x^+ \in X_{ij}^+} d_{G_{ij}}^+(x^+) + \sum_{x^- \in X_{ij}^-}  d_{G_{ij}}^-(x^-) \right).	\label{eqn:e(H_{ij})}
		\end{align}
		Then, by Vizing's theorem for multigraphs (Theorem~\ref{thm:Vizing}), there exists a matching $M_{ij}^H$ in~$H_{ij}$ of size at least $e(H_{ij})/(2\theta d+2) \ge e(H_{ij})/3\theta d$.
		Let $M_{ij}^0$ be the corresponding matching in~$G_{ij}$.
		If $e(M_{ij}^0) \le 16 k^{10}$, then we set $M_{ij}^0$ to be empty. 
		Thus, together with~\eqref{eqn:e(H_{ij})} we have
		\begin{align*}
			e(G_{ij}) \le   \left(  \sum_{x^+ \in X_{ij}^+} d_{G_{ij}}^+(x^+) + \sum_{x^- \in X_{ij}^-}  d_{G_{ij}}^-(x^-) \right) + (3 \theta d) e( M_{ij}^0 ) + 50 k^{10} \theta d   .
		\end{align*}
		Observe that $\Delta\left(\bigcup_{i,j \in [k]} M_{ij}^0\right)\leq k$.
		Apply Lemma~\ref{lem:matchinglemma} for nonempty matchings $M_{ij}^0$ (with $\ell\le k^2$) to obtain $M_{ij} \subseteq M_{ij}^0$ (set $M_{ij}=\emptyset$ if $M_{ij}^0=\emptyset$) such that $\bigcup_{i,j \in [k]}M_{ij}$ is a matching and, for all $i,j \in [k]$, $e(M_{ij}) \ge e(M_{ij}^0) / 2k^2$.
		The result follows. 	
	\end{proof}
	
	Recall that for any directed matching $M$, $V^+(M)$ and $V^-(M)$ are the sets of starting and ending vertices of the directed edges in~$M$, respectively.
	Formally, $V^+(M) =\{v\in V(M): vw\in E(M) \text{ for some }w\in V(M) \}$ and similarly for $V^-(M)$.

	\begin{proposition} \label{prop:extendingmatching}
		Let $G$ be a digraph, $k\in \mathbb{N}$, and $\mathcal{P}=\{V_{ij}:i,j\in [k]\}$ be a partition of~$V(G)$.
		Let $W \subseteq V(G)$.
		Suppose that, for each $i,j \in [k]$ with $i \ne j$, there exist $Y_{ij}^+ \subseteq V_{i*}$, $Y_{ij}^- \subseteq V_{*j}$ and $M_{ij} \subseteq G_{ij}$ such that 
		\begin{enumerate}[label = {\rm(\alph*)}]
			\item $M = \bigcup_{i, j \in [k]} M_{ij}$ is a matching;
			\item both of $\{Y_{ij}^+, V^+(M_{ij}) : i,j \in [k] \}$ and $\{ Y_{ij}^-, V^-(M_{ij}) : i,j \in [k] \}$ are sets of disjoint sets;
			\item \label{itm:extendingmatchingc} for all $y^+ \in Y_{ij}^+$ and $y^- \in Y_{ij}^-$, $d_{G_{ij}}^+(y^+),d_{G_{ij}}^-(y^-) \ge 2\sum_{i,j \in [k]}  ( |Y_{ij}^+| + |Y_{ij}^-| + e(M_{ij}) ) + |W|+1$. 
		\end{enumerate}
		Then $\mathcal{B}(G, \mathcal{P})$ contains a path system~$\mathcal{Q}$ such that the following hold: 
		\begin{enumerate}[label = {\rm(\roman*)}]
			\item \label{itm:extendingmatchingi} $e(\mathcal{Q}_{ij}) = |Y_{ij}^+| + |Y_{ij}^-| + e(M_{ij})$ for all $i,j \in [k]$, where $\mathcal{Q}_{ij}=\mathcal{Q}\cap E(G_{ij})$;
			\item \label{itm:extendingmatchingii} $V(\mathcal{Q})\cap W \subseteq \bigcup_{i,j \in [k]} (Y_{ij}^+ \cup Y_{ij}^- \cup V(M_{ij})) $;
			\item \label{itm:extendingmatchingiii} if $y^+\in Y_{ij}^+\setminus V^-(M)$, then there exists $u\in V_{*j}\setminus W$  such that the single edge~$y^+u$ is a path in~$\mathcal{Q}$, and a similar statement for $y^-\in Y_{ij}^-\setminus V^+(M)$ holds. 
		\end{enumerate}
	\end{proposition}

	\begin{proof}
		Start by setting $\mathcal{Q}=M$ and then for each $y \in Y_{ij}^+$ (and  $y \in Y_{ij}^-$), greedily add an edge $yv$ (and~$vy$, respectively) in~$G_{ij} \setminus W$ such that $v$ avoids all current vertices in~$\mathcal{Q}$  and all vertices in~$\bigcup_{i,j}(Y_{ij}^+\cup Y_{ij}^-)$, which is possible by~\ref{itm:extendingmatchingc}.
		It is clear that we always maintain a path system, and that \ref{itm:extendingmatchingi}--\ref{itm:extendingmatchingiii}  hold by construction. 
	\end{proof}

	\subsection{Proof of Lemma~\ref{lem:balancedpath}}
	
	We now prove Lemma~\ref{lem:balancedpath}.

	\begin{proof}[Proof of Lemma~\ref{lem:balancedpath}]
		We split the proof into several steps.

		\noindent \textbf{Step 1: Defining $G$.}
		By Proposition~\ref{prop:partition-regular-graphs}, $e( G^0_{i*} ) - e(G^0_{*i}) = d(|V_{i*}|-|V_{*i}|)$ for all $i \in [k]$. 
		Apply Proposition~\ref{prop:subgraph} and, without loss of generality,  obtain a subdigraph~$G$ of~$\mathcal{B}(G^0, \mathcal{P})$ such that
		\begin{align}
			e( G ) & \le 	(k-1)d \sum_{i \in [k]}  | |V_{i*}|-|V_{*i}| |/2 
			\overset{\mathclap{\text{\ref{itm:balancedpath3}}}}{\le} k^2  \gamma d n.
			\label{eqn:subgraph2}\\
			e( G_{i*} ) - e(G_{*i}) & = d(|V_{i*}|-|V_{*i}|) \text{ for all $i \in [k]$,} 	\label{eqn:subgraph1}
		\end{align}
		and, for all $i,j \in [k]$ with $i \le j$, $e(G_{ji}) = 0$.
		\smallskip
		
		\noindent \textbf{Step 2: Finding $X_{ij}^+$, $X_{ij}^-$ and $M_{ij}$.}
		For all $i,j \in [k]$, let
		\begin{align}
			X_{ij}^+ = \{v \in V(G) : d_{G_{ij}}^+(v) \ge \gamma^{1/3} d  \} 
			\,\text{ and }\, 
			X_{ij}^- = \{v \in V(G) : d_{G_{ij}}^-(v) \ge \gamma^{1/3} d  \}. \label{eqn:DEFINITION-X-ij}
		\end{align}
		Observe that $X_{ij}^+\subseteq V_{i*}$ and $X_{ij}^-\subseteq V_{*j}$ for all $i,j\in[k]$.
		Note that $\gamma^{1/3} d  \ge \gamma^{1/3} \alpha n \ge 2$ as $1/n\ll\gamma\ll \alpha$. 
		Apply Lemma~\ref{lem:matching} to~$G$ (with $\theta = \gamma^{1/3} $) and obtain a matching~$M_{ij}$ of $G_{ij}$ for~$i,j \in [k]$, such that
		\begin{enumerate}[label = {\rm(\roman*$'$)}]
			\item $e(G_{ij}) \le \sum_{x^+ \in X_{ij}^+}  d_{G_{ij}}^+(x^+) + \sum_{x^- \in X_{ij}^-}  d_{G_{ij}}^-(x^-) +  (6 k^2 \gamma^{1/3}  d) e( M_{ij} ) + 50 k^{10}\gamma^{1/3} d $; \label{itm:Msize}
			\item for all $uv \in M_{ij}$, $u \notin X_{ij}^+$ and $v \notin X_{ij}^-$;
			\item $\bigcup_{i,j \in [k]} M_{ij}$ is a matching. \label{itm:M}
		\end{enumerate}
		By deleting vertices in $X_{ij}^+ \cup X_{ij}^-$ and edges of~$M_{ij}$ if necessary, we may assume that the RHS of~\ref{itm:Msize} is bounded above by~$e(G_{ij})+d$.
		Hence 
		\begin{align*}
			\gamma^{1/3} d ( |X_{ij}^+|+ |X_{ij}^-|+ e( M_{ij} ) ) 
			& \le 
			\sum_{x^+ \in X_{ij}^+} d_{G_{ij}}^+(x^+) + \sum_{x^- \in X_{ij}^-} d_{G_{ij}}^- ( x^- ) + (6 k^2 \gamma^{1/3} d ) e( M_{ij} ) +  50 k^{10}\gamma^{1/3}  d \\
			& \le e(G_{ij})+ d \overset{\text{\eqref{eqn:subgraph2}}}{\le}  k^2  \gamma d n +d \le   \gamma^{2/3} k^{-2}  d^2 /3,
		\end{align*}
		where the last inequality holds as $d \ge \alpha n$ and $1/n \ll \gamma \ll \alpha, 1/k$.
		Therefore, 
		\begin{enumerate}[label = {\rm(\roman*$'$)}, resume]
			\item $\sum_{i,j \in [k]}  \left ( |X_{ij}^+| + |X_{ij}^-| + e(M_{ij})  \right) \le  \gamma^{1/3} d / 3 $.\label{itm:size}
		\end{enumerate}
		
		\noindent \textbf{Step 3: Defining flow networks.}
		For $i,j \in [k]$, let 
		\begin{align*}
			M =\bigcup_{i,j \in [k]} M_{ij}, \quad
			X^+_{i*} = \bigcup_{j \in [k]} \left( X_{ij}^+ \cup V^+(M_{ij}) \right), \text{ and } 
			X^-_{*j} = \bigcup_{i \in [k]} \left( X_{ij}^- \cup V^-(M_{ij}) \right).
		\end{align*}
		Note that $X^+_{i*} \subseteq V_{i*}$ and $X^-_{*j} \subseteq V_{*j}$.
		Let $X^+ = \bigcup_{i \in [k]} X^+_{i*}$ and $X^- = \bigcup_{j \in [k]} X^-_{*j}$.

		We now define a flow network $\mathcal{F} = (F,w,S,T)$ with multiple sources and sinks as follows.
		Let $S= \{s_i: i \in [k]\}$ and $T= \{t_i: i \in [k]\}$.
		Let $V(F) = S \cup T \cup X^+ \cup X^-$. 
		Here we treat $X^+$ and $X^-$ as disjoint, i.e. both $X^+$ and $X^-$ contain a distinct copy of $v$ for any vertex $v\in X^+\cap X^-$.
		For~$v \in V(G)$, we write $v^+$ (and $v^-$) for the copy of~$v$ belonging to~$X^+$ (and $X^-$, respectively).
		We define $E(F)$ as the union of edge-disjoint paths as follows: for each $i,j \in [k]$ with $i < j$, 
		\begin{itemize}
			\item for each $x^+ \in X_{ij}^+$, we add the directed path $s_{i} x^+ t_{j}$;
			\item for each $x^- \in X_{ij}^-$, we add the directed path $s_{i} x^- t_{j}$; 
			\item for each $e =  uv \in M_{ij}$, we add the directed path $s_{i} u^+ v^- t_{j}$;
			\item every edge and vertex has capacity $1$, i.e.\ $w(e) = 1$ for all $e \in E(F)$ and $w(v) = 1$ for all $v \in V(F) \setminus (S \cup T) = X^+ \cup X^-$.
		\end{itemize}
		We modify the flow network slightly if \ref{item:size-condition-V-ij} or \ref{item:degree-condition-v0} holds. Let
		\begin{align*}
			F_0=\begin{cases}
				F - \{v_0^+ \} - \{s_{i_0} v_0^-\} - E_F(X^+,v_0^-)&\text{ if }v_0\in V_{i_0j_0}\text{ exists satisfying \ref{item:degree-condition-v0}},\\
				F&\text{ otherwise.}
			\end{cases}
		\end{align*}
		
		We further remove $v_0^-$ from~$F_0$ if $v_0^-$ has no inneighbour in~$F_0$.
		Write $w_0$ for the capacities on~$F_0$ inherited from~$F$ and define $\mathcal{F}_0 = (F_0,w_0,S,T)$. 
		By the definition of~$\mathcal{F}_0$, note for later that
		\begin{align}
			\text{if there is any flow through }v_0^-\text{, then it must be via some edge }s_iv_0^-\text{ with }i\ne i_0. \label{fact:incoming-edges-v0}
		\end{align}

		Given a flow~$f$ on~$\mathcal{F}$ or $\mathcal{F}_0$, for each $i,j \in [k]$, let $f_{ij}$ be the sum of flow over all edges from $s_i \cup X^+_{i*}$ to $ t_j \cup X^-_{*j}$, that is,  
		\begin{align}
			f_{ij} = \sum_{e \in E_F(s_i \cup X^+_{i*},  t_j \cup X^-_{*j})} f(e).\label{eqn:definition-of-f-ij}  
		\end{align}
		Note that the total amount of flow going out of $s_i$ (all going to $X^+$) is $\sum_{j\in[k]}f_{ij}$, and the total amount of flow into $t_j$ is $\sum_{i\in[k]}f_{ij}$.
		We now reduce the lemma to the following claim about fractional flows. 
		
		\begin{claim} \label{clm:fractionalflow}
			To prove the lemma, it suffices to find a fractional flow~$f$ on~$\mathcal{F}_0$ such that 
			\begin{align}
				\sum_{i,j \in [k]} \max\{ e(G_{ij})/d - f_{ij}, 0 \} < 1. \label{eqn:flow}
			\end{align}
		\end{claim}
		
		\begin{proofclaim} 
			Assume $f$ is as given in the claim and we wish to prove the lemma.
			We may assume that $f_{ij} \le   e(G_{ij})/d$ for all $i,j \in [k]$.\footnote{Otherwise we could suitably decrease flow along paths $P$ with positive flow from any source to any sink that uses an edge from $E_F(s_i\cup X_{i*}^+, t_j\cup X_{*j}^-)$, and this can be done independently for each $i,j\in[k]$, since the $E_F(s_i\cup X_{i*}^+, t_j\cup X_{*j}^-)$ are disjoint over all $i,j\in[k]$.}
			We define a new flow network $\mathcal{F}^* = (F^*, w^*, s^*,t^*)$ with a single source~$s^*$ and sink~$t^*$ as follows. 
			We obtain $\mathcal{F}^*$ from~$\mathcal{F}_0$ by adding the new vertices~$s^*$ and~$t^*$ to~$V(F_0)$ and, for each $i \in [k]$,
			\begin{itemize}
				\item adding the edge $s^* s_i$ of capacity $\max \{|V_{i*}| - |V_{*i}|, 0\}$;
				\item adding the edge $t_i t^*$ of capacity $\max \{|V_{*i}| - |V_{i*}|, 0\}$;
				\item adding an edge $t_i s_i$ of infinite capacity; 
				\item giving the vertices in $S \cup T$ infinite capacity.
			\end{itemize}
			
			We will use $f$ to find an integer flow in~$\mathcal{F}^*$ of value $ \sum_{i \in [k]} \left| |V_{i*}| - |V_{*i}|\right|/2$. 
			Then, by looking at the vertices through which there is non-zero flow, we will define the sets $Y_{ij}^+$, $Y_{ij}^+$, and~$M_{ij}^*$ satisfying the properties for Proposition~\ref{prop:extendingmatching}, which will give the desired path system~$\mathcal{Q}$.
			
			First we find a fractional flow on~$\mathcal{F}^*$ as follows.
			Let $\mathcal{F}^+ = (F^+, w^+,s^*,t^*)$ be obtained from $\mathcal{F}^*$ by adding, for each $i,j \in [k]$, an edge $s_i t_j$ with edge capacity~$1$. 
			Define a flow~$f^+$ on $\mathcal{F}^+$ such that, for all $i,j \in [k]$, 
			\begin{itemize}
				\item $f^+(e) = f(e)$ for all $e \in E(F_0)$;
				\item $f^+(s^* s_i)  = \max \{|V_{i*}| - |V_{*i}|, 0\} $;
				\item $f^+(t_j t^*) = \max \{|V_{*j}| - |V_{j*}|, 0\}$;
				\item $f^+(t_i s_i) = \min  \{ e(G_{i*})/d, e(G_{*i})/d\}$;
				\item $f^+(s_i t_j) = e(G_{ij})/d - f_{ij}\ge 0$.
			\end{itemize}
			One can check that $f^+$ is indeed a fractional flow\footnote{Since $f$ is a fractional flow on~$\mathcal{F}_0$, it suffices to examine only $s_i$ and $t_i$ for $i\in[k]$. For $s_i$, the amount of flow going out of $s_i$ is $\sum_{j\in[k]}f_{ij}$ (from $s_i$ to $X^+$) and $\sum_{j\in[k]}\left(e(G_{ij})/d-f_{ij}\right)$ (from $s_i$ to $T$), and the flow going into $s_i$ is $\min\{e(G_{i*})/d,e(G_{*i})/d\}$ (from $t_i$ to $s_i$) and $\max\{|V_{i*}|-|V_{*i}|,0\}$ (from $s^*$ to $s_i$). By noting that~\eqref{eqn:subgraph1} gives $e(G_{i*})/d-e(G_{*i})/d=|V_{i*}|-|V_{*i}|$, the total contribution is 0. For $t_i$, the calculation is similar.} on $\mathcal{F}^+$ with value $\sum_{i \in [k]} \left| |V_{i*}| - |V_{*i}|\right|/2$. 
			By applying Proposition~\ref{prop:flow-reduction} to reduce the flow in $s_it_j$ to~$0$ for all $i,j\in[k]$, we obtain a fractional flow on~$\mathcal{F}^*$ of value 
			\begin{align*}
				r & \geq \frac12\sum_{i \in [k]} \left| |V_{i*}| - |V_{*i}|\right| - \sum_{i,j \in [k]} f^+(s_it_j)
				= \frac12 \sum_{i \in [k]} \left| |V_{i*}| - |V_{*i}|\right| - \sum_{i,j \in [k]} \left( \frac{e(G_{ij})}d - f_{ij} \right)\\
				& \overset{\text{\eqref{eqn:flow}}}{>}  \frac12\left(\sum_{i \in [k]} \left| |V_{i*}| - |V_{*i}|\right|\right)-1.   
			\end{align*}

			Note that all edge capacities are integral, so Theorem~\ref{thm:maxflow} implies there exists an integral flow~$f^*$ on~$\mathcal{F}^*$ of value at least $\lceil r \rceil =  \sum_{i \in [k]} \left| |V_{i*}| - |V_{*i}|\right|/2$. 
			Moreover,  since the cuts $E_{F^*}(s^*, V(F^*) \setminus \{s^*\})$ and $E_{F^*}( V(F^*) \setminus \{
			t^*\}, t^*)$ have capacity $\sum_{i \in [k]} \left| |V_{i*}| - |V_{*i}|\right|/2$, Theorem~\ref{thm:maxflow} implies that $f^*$ saturates all the edges $s^*s_i$ and $t_jt^*$ for all $i,j \in [k]$.

			We now define $Y_{ij}^+ \subseteq X_{ij}^+$, $Y_{ij}^-\subseteq X_{ij}^-$ and~$M^*_{ij} \subseteq M_{ij}$ as follows. 
			Initially, set $Y_{ij}^+ = Y_{ij}^- = M_{ij}= \emptyset$ for all $i,j \in [k]$. 
			For each $x^+\in X^+$ with $f^*(x^+) = 1$, since $f^*$ is an integral flow, there exist unique $i,j\in[k]$ such that either $f^*(s_i x^+)=f^*(x^+t_j)=1$ or $f^*(s_i x^+)=f^*(x^+z)=f^*(zt_j)=1$ for some (unique) $z\in X^-$. In the former case we add~$x^+$ to~$Y_{ij}^+$ while we add~$x^+z$ into~$M_{ij}^*$ for the latter. In a similar way, for $x^- \in X^-$ with $f^*(x^-) = 1$, we either add $x^-$ to the set $Y_{ij}^-$ or add an edge ending at $x^-$ into $M_{ij}^*$.
			Note that 
			\begin{align}
				|Y_{ij}^+| + |Y_{ij}^-| + e (M^*_{ij})  =\sum_{e\in E_{F^*}(s_i\cup X_{i*}^+,t_j\cup X_{*j}^-)}f^*(e)     
				\eqqcolon f^*_{ij}. \label{eqn:flow2}
			\end{align}
			Clearly $ M^* : = \bigcup_{i,j \in[k]} M^*_{ij} \subseteq M$ is a matching by~\ref{itm:M}.
			Both of $\{Y_{ij}^+, V^+(M_{ij}^*) : i,j \in [k] \}$ and $\{ Y_{ij}^-, V^-(M_{ij}^*) : i,j \in [k] \}$ are sets of disjoint sets as each $x \in X^+ \cup X^-$ has vertex capacity of one.
			Set 
			\begin{align*}
				W =
				\begin{cases}
					\bigcup_{i,j \in [k] \colon i \ne j} V_{ij}  & \text{if $\sum_{i,j \in [k] \colon i \ne j} |V_{ij}| \le k^2 \gamma n$,}\\
					\emptyset & \text{otherwise.}
				\end{cases}
			\end{align*}
			For all $y^+ \in Y_{ij}^+ \subseteq X_{ij}^+$, we have 
			\begin{align*}
				d_{G_{ij}}^+(y^+) & \overset{\mathclap{\text{\eqref{eqn:DEFINITION-X-ij}}}}{\ge}   \gamma^{1/3} d \overset{\mathclap{\text{\ref{itm:size}}}}{\ge} 
				2\sum_{i,j \in [k]} \left( |X_{ij}^+| + |X_{ij}^-| + e(M_{ij})  \right)+ \gamma^{1/3} d/3
				\\   & 
				\ge    
				2\sum_{i,j \in [k]} \left( |Y_{ij}^+| + |Y_{ij}^-| + e(M^*_{ij}) \right)+ |W|+1
			\end{align*}
			since $k^2\gamma n+1\le 2k^2\gamma n\le \gamma^{1/3}d/3$ as $\gamma \ll \alpha \le d/n$, and a similar inequality holds for all $y^- \in Y_{ij}^-$.
			We apply Proposition~\ref{prop:extendingmatching} and obtain a path system~$\mathcal{Q}$ of~$G$ such that 
			\begin{enumerate}[label = {\rm(\roman*$''$)}]
				\item \label{itm:flow3} $e(\mathcal{Q}_{ij}) = |Y_{ij}^+| + |Y_{ij}^-| + e(M^*_{ij}) \overset{\mathclap{\text{\eqref{eqn:flow2}}}}{=} f^*_{ij}$ for all $i,j \in [k]$;
				\item  $V(\mathcal{Q})\cap W \subseteq \bigcup_{i,j \in [k]} (Y_{ij}^+ \cup Y_{ij}^- \cup V(M^*_{ij})) $;
				\label{item:path-system-2}
				\item if $y^+\in Y_{ij}^+\setminus V^-(M^*)$, then there exists $u\in V_{*j}\setminus W$  such that the single edge~$y^+u$ is a path in~$\mathcal{Q}$, and a similar statement for $y^-\in Y_{ij}^-\setminus V^+(M^*)$ holds.\label{item:path-system-3}
			\end{enumerate}

			Consider $i \in [k]$.
			Recall that $N^-_{F^*}(s_i) = \{s^*, t_i\}$, $ N^+_{F^*}(s_i) \subseteq X^+_{i*} \cup X^-$, and each vertex in $X^+$ has capacity~one. 
			We have
			\begin{align*}
				f^*(s^*s_i) + f^*(t_is_i) 
				& = \sum_{u \in N^-_{F^*}(s_i)} f^*( u s_i)
				= \sum_{v \in N^+_{F^*}(s_i)} f^*(  s_i v )\\
				&= \sum_{v \in X^+_{i*} \cup X^-} f^*(  s_i v )
				= \sum_{j \in [k]} \sum_{e \in E_{F^*} (s_i \cup X^+_{i*},  t_j \cup X^-_{*j})} f^*(e)
				= \sum_{j \in [k]} f^*_{ij}
				\overset{\mathclap{\text{\ref{itm:flow3}}}}{=}  \sum_{j \in [k]} e(\mathcal{Q}_{ij}) 
			\end{align*}
			and similarly we have $	f^*(t_it^*) + f^*(t_is_i) = \sum_{j \in [k]} e(\mathcal{Q}_{ji}) $.
			Since $f^*$  saturates the edges $s^*s_i$ and~$t_it^*$, we deduce that,
			\begin{align*}
				\sum_{j \in [k]} e(\mathcal{Q}_{ij}) - \sum_{j \in [k]} e(\mathcal{Q}_{ji})  &=f^*(s^*s_i)-f^*(t_it^*)
				= w^*(s^*s_i)-w^*(t_it^*)
				\\
				&=\max \{|V_{i*}| - |V_{*i}|, 0\}-\max \{|V_{*i}| - |V_{i*}|, 0\}
				= |V_{i*}|-|V_{*i}|
			\end{align*}
			implying that $\mathcal{Q}$ is $\mathcal{P}$-balanced, as required. 
			By applying Proposition~\ref{prop:subgraph} (with $d = 1$ and $G^0=\mathcal{Q}$), we may further assume that 
			\begin{align}
				e(\mathcal{Q})  \le (k-1) \sum_{i \in [k]} | |V_{i*}|-|V_{*i}| |/2 
				\overset{\mathclap{\text{\ref{itm:balancedpath3}}}}{\le} k^2  \gamma n
				. \label{eqn:flow4}
			\end{align}
			We now check the moreover statement of the Lemma~\ref{lem:balancedpath}. 
			If $\sum_{i,j \in [k] \colon i \ne j} |V_{ij}| > k^2\gamma n$ (that is, \ref{item:size-condition-V-ij} holds), then \eqref{eqn:flow4} implies that $\bigcup_{i,j \in [k] \colon i \ne j} V_{ij} \setminus V(\mathcal{Q}) \ne \emptyset$, so $\mathcal{Q}$ is non-trivial. 
			Suppose that $\sum_{i,j \in [k] \colon i \ne j} |V_{ij}| \le  k^2 \gamma n$ and there exists $v_0\in V_{i_0j_0}$ satisfying~\ref{item:degree-condition-v0}. 
			If $v_0 \notin V(\mathcal{Q})$, then $\mathcal{Q}$ is non-trivial.
			So suppose that $v_0 \in V(\mathcal{Q})$. 
			Since $v_0 \in W$,  by~\ref{item:path-system-2} we must have $v_0 \in Y_{ij}^+ \cup Y_{ij}^- \cup V(M_{ij}^*)$ for some $i,j \in [k]$. 
			This means that there is a flow through $v_0^-$ as $v_0^+ \notin V(F^*)$.
			By~\eqref{fact:incoming-edges-v0}, $v_0 \in Y_{ij_0}^-$ for some $i \ne i_0$. 
			Also, $v_0^+\notin V(F^*)$ implies that $v_0\notin V^+(M^*)$.
			By~\ref{item:path-system-3}, we deduce that there exists $u \in V_{i*} \setminus W = V_{ii}$ such that $uv_0$ is a path in~$\mathcal{Q}$.
			Note that $uv_0$ is a path from~$V_{ii} \subseteq V_{*i}$ to~$V_{i_0j_0} \subseteq V_{i_0 *}$. 
			Hence $\mathcal{Q}$ is non-trivial, as required. 
		\end{proofclaim}
		
		\noindent \textbf{Step 4: Define a distribution function on $X^+ \cup X^-$.}
		For each $i \in [k]$, $x^+ \in X^+_{i*}$ and $x^- \in X^-_{*i}$, $j \in [k] \setminus \{ i\}$, we set
		\begin{align*}
			p_j({x^+})	 = \max \{ d_{G_{ij}}^+(x^+)/d , 6 k^2 \gamma^{1/3}\} \,\,\text{ and }\,\,
			p_j({x^-})	 = \max \{d_{G_{ji}}^-(x^-)/d , 6 k^2 \gamma^{1/3}\}.
		\end{align*}
		Given $x^+ \in X^+_{i*}$, we `view' $p_j({x^+})$ to be flow through~$x^+$ from~$s_i$ to~$t_j$, and similarly for $p_j({x^-})$.

		\begin{claim} \label{clm:probability}
			For each $i \in [k]$ and $x^+ \in X^+_{i*}$, 
			\begin{align*}
				\sum_{j \in [k] \setminus \{i\} } p_j({x^+}) \le 1 - \frac{d_{G^0}^+(x,V_{*i})}{d} + 6 k^3 \gamma^{1/3} < 1
			\end{align*}
			and a similar statement holds for each $j \in [k]$ and $x^- \in X^-_{*j}$. 
		\end{claim}
		\begin{proofclaim}
			We will only consider the case when $x^+ \in X^+_{i*}$ as the other case when $x^- \in X^-_{*j}$ can be proved similarly.
			Since $G$ is a subgraph of~$G^0$, we have 
			\begin{align*}
				\sum_{j \in [k]\setminus\{i\} } d_{G_{ij}}^+(x^+) \le \sum_{j \in [k]\setminus\{i\} } d_{G_{ij}^0}^+(x^+) = d-d_{G^0}^+(x,V_{*i}) .
			\end{align*}
			Hence
			\begin{align*}
				\sum_{j \in [k]\setminus \{i\} } p_j({x^+}) 
				& \le \sum_{j \in [k]\setminus \{i\} }  \left( \frac{d_{G_{ij}}^+(x^+)}d + 6 k^2 \gamma^{1/3} \right)
				\le  \frac{\sum_{j \in [k]\setminus \{i\}} d_{G_{ij}}^+(x^+)} d + 6 k^3 \gamma^{1/3}\\
				& \le 1 - \frac{d_{G^0}^+(x,V_{*i})}{d} + 6 k^3 \gamma^{1/3} <1,
			\end{align*}
			where the last inequality holds as $d_{G^0}^+(x,V_{*i})\ge d/k$ by~\ref{item:good-degree-condition} and $\gamma\ll 1/k$.
		\end{proofclaim}

		\noindent \textbf{Step 5: Defining fractional flow on $\mathcal{F}_0$.}
		We first define the fractional flow~$f$ on~$\mathcal{F}$ as follows.
		For each $i,j \in [k]$, we do the following
		\begin{itemize}
			\item for each $x^+ \in X_{ij}^+ \subseteq V_{i*}$, add a flow through $s_{i} x^+ t_{j} $ of value~$p_j(x^+)$; 
			\item for each $x^- \in X_{ij}^- \subseteq V_{*j}$, add a flow through $s_{i} x^- t_{j}$ of value~$p_i(x^-)$;
			\item for each edge $e=uv \in M_{ij}$, add a flow through $s_{i} u^+  v^- t_{j}$ of value~$6 k^2 \gamma^{1/3}$.
		\end{itemize}
		Our construction and Claim~\ref{clm:probability} imply that each vertex $v \in X^+ \cup X^- $ has a flow of value at most~$1$ through it. 
		Thus $f$ is indeed a fractional flow on~$\mathcal{F}$.
		For each $i,j \in [k]$, 
		\begin{align}
			d f_{ij} 
			& 
			\overset{\mathclap{\text{\eqref{eqn:definition-of-f-ij}}}}{=}
			d \left( \sum_{x^+ \in X_{ij}^+} p_j(x^+) + \sum_{x^- \in X_{ij}^-} p_i(x^-) + \sum_{uv \in M_{ij}} 6 k^2 \gamma^{1/3} \right) \nonumber \\
			& \ge \sum_{x^+ \in X_{ij}^+} d_{G_{ij}}^+(x^+) +  \sum_{x^- \in X_{ij}^-} d_{G_{ij}}^-(x^-) +   (6 k^2 \gamma^{1/3}  d) e( M_{ij} )
			\overset{\mathclap{\text{\ref{itm:Msize}}}}{\ge} e(G_{ij})  - 50 k^{10} \gamma^{1/3}d. \nonumber
		\end{align}
		Therefore, 
		\begin{align*}
			\sum_{i,j \in [k]} \max\{ e(G_{ij})/d - f_{ij}, 0 \}	\le 50 k^{12} \gamma^{1/3} <1 .
		\end{align*}
		If there is no $v_0$ satisfying~\ref{item:degree-condition-v0}, then $\mathcal{F}_0 = \mathcal{F}$ and we are done by~Claim~\ref{clm:fractionalflow}.
		If such a $v_0$ exists, then recall that $F_0$ is obtained from~$F$ removing the vertex~$v_0^+$,  the edge $s_{i_0} v_0^-$ and edges in~$E_F(X^+,v_0^-)$.
		By repeated application of Proposition~\ref{prop:flow-reduction}, we obtain a flow in $\mathcal{F}_0$ whose value is lower than that of $f$ by at most
		\begin{align*}
			& \sum_{j \in [k]\setminus \{ i_0 \} } p_j(v_0^+) + p_{i_0}(v_0^-) + 2\cdot   6 k^2 \gamma^{1/3}\\
			& \overset{\mathclap{\text{Claim~\ref{clm:probability}}}}{\le}  1 - \frac{d_{G^0}^+(v_0,V_{*i_0})}{d} + 6 k^3 \gamma^{1/3} + \left( \frac{d^-_{G_{i_0j_0}}(v_0)}{d} + 6 k^2 \gamma^{1/3} \right) + 	 12k^2  \gamma^{1/3}
			\\
			&\le  1 - \frac{d_{G^0}^+(v_0,V_{*i_0})}{d}+\frac{d^-_{G^0}(v_0,V_{i_0*})}{d} +24k^3\gamma^{1/3}
			\\
			&  \overset{\mathclap{\text{\ref{item:degree-condition-v0}}}}{\le}  1 -  100 k^{12} \gamma^{1/3}+24k^3\gamma^{1/3}< 1-50k^{12}\gamma^{1/3}.
		\end{align*}
		Therefore, $\sum_{i,j \in [k]} \max\{ e(G_{ij})/d - f_{ij}, 0 \}$ is still less than~$1$. 
		Hence we are done by Claim~\ref{clm:fractionalflow}.
	\end{proof}

	
	\section{Proof of Lemma~\ref{lma:regdigraph}}
	\label{sec:pf_regdigraph}

	We begin with an outline of the proof of the Lemma~\ref{lma:regdigraph}.
	Let $G$ and $\mathcal{P}$ be the digraph and vertex partition as in the statement of Lemma~\ref{lma:regdigraph}.
	Suppose, contrary to the lemma, that there is no non-trivial $\mathcal{P}$-balanced path system.

	Consider distinct $i,j \in [k]$.
	Note that if there exist two vertex-disjoint edges $e_1 \in E(G_{ij})$ and $e_2 \in E(G_{ji})$, then $\{e_1,e_2\}$ forms a non-trivial $\mathcal{P}$-balanced path system.
	Thus, if $e(G_{ij}), e(G_{ji}) \ge 3$, then we may assume that there exists a vertex $ w_{ij} \in V_{ii} \cup V_{jj}$ that is contained in all edges of~$G_{ij} \cup G_{ji}$. 
	Let $H$ be the oriented graph on~$[k]$ where $ij \in E(H)$ if $e(G_{ij}), e(G_{ji}) \ge 3$ and $w_{ij} \in V_{jj}$; we give weight $w(ij):=|E(G_{ij}\cup G_{ji})|$ to $ij$.\footnote{Note that it is convenient to introduce the weight $w(ij)$ (and later $w(i)$) for the sketch of proof; these are not used in the actual proof.}  It turns out that the underlying undirected graph of $H$ is acyclic (see Claim~\ref{clm:cyclesize}) and that total weight of edges in $H$ is almost equal to~$|\mathcal{B}(G, \mathcal{P})|$ (see~\eqref{eqn:e(G*)}).
	Set $w(i)=\sum_{j\in [k]}(w(ij)-w(ji))$. 
	
	We focus on the $V_{ii}$ that are relatively small: let $c_i=qd+1-|V_{ii}|$ (where we set $q=2$ if $G$ is an oriented graph and $q=1$ otherwise) and after relabelling indices assume $ |\{i \in [k]: c_i\geq 0\}| = [k_0]$. We will lower and upper bound $\sum_{i\in[k_0]}w(i)$.
	For the lower bound, we note that if $c_i>0$ then every vertex of $V_{ii}$ has at least one edge from outside~$V_{ii}$, and so we can lower bound $w(i)$ in terms of $c_i$ for each $i\in [k_0]$ (see \eqref{eqn:e(G[V'_i,barV_i])}).
	For the upper bound, we note that $w(ij)$ can be upper bounded using~\ref{item:degree-is-large-inside-diagonal} (see~\eqref{eqn:e(Gij)+e(Gji)}). Then, as $H$ is acyclic (so has few edges), we are able to find a good upper bound for $\sum_{i\in[k_0]}w(i)$ by considering connected components (i.e.~trees) in $H[[k_0]]$. Combining the lower and upper bounds, we obtain $\sum_{i\in[k_0]}c_i\leq k-k_0-1$ (see \eqref{eqn:sum-of-components-vertices-k0}) whereas \ref{item:relation-n-d-k} implies $\sum_{i\in[k_0]}c_i>k-k_0$, a contradiction.

	Recall that for a digraph~$G$ and $A,B \subseteq V(G)$ not necessarily disjoint, we write $E_G(A,B):= \{ab \in E(G): a \in A, \; b \in B \}$ and $e_G(A,B) := |E_G(A,B)|$.
	We write $E_G(A)$ and $e_G(A)$ for $E_G(A,A)$ and~$e_G(A,A)$, respectively.

	\begin{proof}[Proof of Lemma~\ref{lma:regdigraph}]
		Assume $G$ and $\mathcal{P}$ satisfy \ref{item:each-part-is-at-least-d-over-2}--\ref{item:degree-is-large-inside-diagonal} in the statement of Lemma~\ref{lma:regdigraph}. 
		Let $q =2$ if $G$ is an oriented graph and $q =1$ otherwise, so \ref{item:relation-n-d-k} says that $n<(qd + 1)k$. 
		We write $V_i$ instead of~$V_{ii}$ for $i \in [k]$, so that, for all distinct $i,j\in[k]$, $G_{ij}$ becomes the bipartite digraph with vertex classes $V_i$ and~$V_j$ and edges given by the edges of $G$ from~$V_{i}$ to~$V_{j}$.
		In this context (where $V_{ij}=\emptyset$ for all distinct $i,j\in[k]$), $\mathcal{Q}$ is a $\mathcal{P}$-balanced path system if, for each $i\in[k]$, the number of edges of~$\mathcal{Q}$ going into $V_i$ is equal to the number of edges of~$\mathcal{Q}$ leaving~$V_i$. Also, a $\mathcal{P}$-balanced path system is non-trivial if it has at least one path whose endpoints lie in distinct parts of the partition $\mathcal{P}=\{V_i:i\in[k]\}$. Now suppose for a contradiction that $G$ does not contain any non-trivial $\mathcal{P}$-balanced path system with at most $k$ edges.

		\begin{claim} \label{clm:cyclesize}
			Let $H_0$ be a digraph on~$[k]$ such that $ij \in E(H_0)$ if $e(G_{ij}) \ge 3$ and $i\neq j$.
			Then $H_0$ does not have any directed cycle of length at least~$3$.
		\end{claim}
		
		\begin{proofclaim}
			Suppose to the contrary that $C$ is such a cycle in $H_0$. 
			For each $ij \in E(C)$,
			we pick an edge~$e_{ij}\in E(G_{ij})$ with $ij \in E(C)$ and call the resulting subdigraph~$G'$.
			Note that $G'$ is $\mathcal{P}$-balanced and has at most $k$ edges; moreover it is either a cycle or a path system. 
			If $G'$ is a cycle, then pick an $ij\in E(C)$, and find $e_{ij}'\in E(G_{ij})$ with $e_{ij}'\neq e_{ij}$. Write $G''=(G'-\{e_{ij}\})\cup \{e_{ij}'\}$ if $G'$ is a cycle and $G''=G'$ otherwise.
			Note that $G''$ is a $\mathcal{P}$-balanced path system. 
			If $G''$ is a single path, then there exists $t\in [k]$ such that the endpoints of~$G''$ lies in~$V_{t}$.
			Then, pick an $i'j'\in E(C)$ with $i',j'\neq t$, and find $e_{i'j'}'\in E(G_{i'j'})$ with $e_{i'j'}'\neq e_{ij}, e_{i'j'}$. 
			Write $G'''=(G''-\{e_{i'j'}\})\cup \{e_{i'j'}'\}$ if $G''$ is a single path and $G'''=G''$ otherwise.
			Note that $G'''$ consists of at least two vertex-disjoint paths and that for any path in $G'''$, the endpoints lie in distinct parts of the partition $\mathcal{P}=\{V_{i}:i\in[k]\}$, so $G'''$ is a non-trivial $\mathcal{P}$-balanced path system with at most $k$ edges, a contradiction.
		\end{proofclaim}

		\begin{claim} \label{clm:sizedifference}
			For all distinct $i,j \in [k]$, if $e(G_{ij})\le 2$, then $e(G_{ji})\le 2k^2$.
		\end{claim}
		
		\begin{proofclaim}
			Suppose to the contrary and without loss of generality that $e(G_{21})\le 2$ and $e(G_{12})\ge 2k^2+1$.
			Let $H_0$ be as defined in Claim~\ref{clm:cyclesize}, so $12 \in E(H_0)$ and $21\notin E(H_0)$.
			Let $A$ be the set of $i\in [k]$ such that $H_0$ contains a directed path from~$1$ to~$i$ starting with the edge $12$, where we a priori allow $1 \in A$ (i.e.~resulting from a closed path). 
			Let $B=[k]\setminus A$.
			It is clear from the definition that $2\in A$, and by Claim~\ref{clm:cyclesize} and since $21\notin E(H_0)$ we have $1\in B$. 
			Note that by the definition of~$A$, there are no edges in~$E_{H_0}(A,B)$ so that $e(G_{ij})\le 2$ for all $ij\in A\times B$. 
			Together with the regularity of~$G$, we have 
			\begin{align*}
				2k^2\ge \sum_{ij\in A\times B}e(G_{ij})=\sum_{ij\in B\times A} e(G_{ij}) \ge e(G_{12})\ge 2k^2+1,  
			\end{align*}
			which is a contradiction.
		\end{proofclaim}

		Let $H^*$ be the graph on~$[k]$ such that $ij \in E(H^*)$ if $e(G_{ij}), e(G_{ji}) \ge 3$, i.e. $H^*$ is the graph obtained from $H_0$ by deleting all the edges which do not lie in a directed $2$-cycle and by making each directed $2$-cycle into an undirected edge.
		By Claim~\ref{clm:cyclesize}, $H^*$ is acyclic. 
		Consider distinct $i,j \in [k]$.
		Note that there do not exist disjoint edges $e_1, e_2$ with $e_1 \in E(G_{ij})$ and $e_2 \in E(G_{ji})$ because otherwise $\{e_1,e_2\}$ would form a non-trivial $\mathcal{P}$-balanced path system. Thus if $ij \in E(H^*)$ (that is, $e(G_{ij}) , e(G_{ji})\ge 3$), then there exists a vertex $w_{ij}\in V_{i} \cup V_j$ that is contained in all edges of~$G_{ij} \cup G_{ji}$. So, we have 
		\begin{align*}
			e(G_{ij}) +e(G_{ji})  = d^+_{G_{ij} \cup G_{ji}} (w_{ij}) + d^-_{G_{ij} \cup G_{ji}}(w_{ij}) \overset{\mathclap{\text{\ref{item:degree-is-large-inside-diagonal} }}}{\le } ( 2-1/k ) d . 
		\end{align*}
		If $ij \not\in E(H^*)$, then $\min\{e(G_{ij}), e(G_{ji})\}\le 2$ and by Claim~\ref{clm:sizedifference}, we obtain
		\begin{align}
			e(G_{ij}) +e(G_{ji})\leq 2k^2+2\leq (  2-1/k)d \label{eqn:e(Gij)+e(Gji):2}    
		\end{align}
		as $k\ge 2$ and $d>165k^5$. 
		Hence we have 
		\begin{align}
			e(G_{ij}) +e(G_{ji}) \le (  2-1/k) d \text{ for all distinct $i ,j \in[k]$} .
			\label{eqn:e(Gij)+e(Gji)}    
		\end{align} 
		Recall that $q =2$ if $G$ is an oriented graph, and $q =1$ otherwise. Note that, if $|V_i|\le qd+1$ for some $i\in[k]$, then
		\begin{align}
			2kd  & 	
			\overset{\mathclap{\text{\eqref{eqn:e(Gij)+e(Gji)}}}}{\ge }  	\sum_{j \in [k] \setminus \{i\}}e(G_{ij}) 
			= 	
			d|V_i|  - e(G_{ii}) 
			\ge d|V_i|  - \frac{|V_i| (|V_i|-1) }{q} \nonumber \\
			&
			= \frac{q d +1 - |V_i|}{q} |V_i| 
			\overset{\mathclap{\text{ \ref{item:each-part-is-at-least-d-over-2} }}}{\ge }   \frac{d ( q d +1 - |V_i|)}{2q}
			\ge \frac{d ( q d +1 - |V_i|)}{4}
			. \nonumber
		\end{align}
		Hence, for all $i\in [k]$,
		\begin{align}
			|V_i| & \ge qd+1-8k.
			\label{eqn:sizeV_i}
		\end{align}
		
		We now orient~$H^*$ to obtain an oriented graph~$H$ on~$[k]$ as follows. 
		Recall that if $ij \in E(H^*)$, then there exists a vertex $w_{ij} \in V_i \cup V_j$ contained in all edges of~$E(G_{ij}) \cup E(G_{ji})$.
		We orient from $i$ to $j$ if $w_{ij} \in V_j$.

		Let $W$ be the set of $w_{ij}$ and $W_i = W \cap V_i$ for all $i \in [k]$. 
		Let $V_i'=V_i\setminus W_i$.
		Note that 
		\begin{align}
			|W_i| \le k-1 \text{ and so } |V_i'|= |V_i|-|W_i|\overset{\mathclap{\text{\eqref{eqn:sizeV_i}}}}{\ge }qd-9k. \label{eqn:|W_i|}
		\end{align}
		Recall that $\mathcal{B}(G, \mathcal{P})=\bigcup_{i,j\in[k]:\,i\neq j}G_{ij}$. 
		Let $G^* = \bigcup_{i,j \in [k]}G^*_{ij}$ where $G^*_{ij}$ is the subdigraph of $\mathcal{B}(G, \mathcal{P})$ consisting of the edges (in both directions) between $V_i'$ and $w_{ij}$ if $ij\in E(H)$ and $G^*_{ij} = \emptyset$ otherwise.\footnote{i.e. $G^*_{ij}=G_{ij}[V_i'\cup w_{ij}]\cup G_{ji}[V_i'\cup w_{ij}]$ if $ij\in E(H)$ and $G^*_{ij}=\emptyset$ otherwise.} Note that $G^*_{ij}$ is a subdigraph of $G_{ij}\cup G_{ji}$ whenever $ij\in E(H)$ and $G^*_{ij}=\emptyset$ otherwise. 
		Therefore,
		\begin{align}
			e( G^*_{ij} ) 
			\begin{cases}
				\le e(G_{ij} \cup G_{ji})  & \text{ if $ij \in E(H)$;} \\
				= 0 & \text{ if $ij \notin E(H)$.}
			\end{cases}
			\label{eqn:H_ij}
		\end{align}
		Hence, we deduce that
		\begin{align}
			e (\mathcal{B}(G, \mathcal{P}) )- e(G^*) & \le \sum_{ij \notin E(H^*)} \left( e(G_{ij})+ e(G_{ji}) \right)  + \sum_{i,j \in [k] : i \ne j } |W_i||W_j| \nonumber \\
			& \overset{\mathclap{\text{\eqref{eqn:e(Gij)+e(Gji):2},\eqref{eqn:|W_i|}}}}{\le} \quad\binom{k}2 (2k^2+2) + k (k-1)^3 
			< 3 k^4
			.
			\label{eqn:e(G*)}
		\end{align}
		
		For $i \in [k]$, let $\overline{V_i} = V(G) \setminus V_i$.
		Since all edges in~$G^*$ contain a vertex in~$W$ and $G^*[W] \cup \bigcup_{i \in [k]}G^*_{ii}$ is empty, we have for all $i \in [k]$, 
		\begin{align}
			\sum_{j\in[k]}e(G^*_{ij})=\sum_{j\in[k],\,ij\in E(H)}\bigg(e_{G^*}(V_i',w_{ij})+e_{G^*}(w_{ij},V_i')\bigg)=e_{G^*}(V_i',\overline{V_i})+e_{G^*}(\overline{V_i},V_i')
			\label{eqn:G^*[V_i', W_j]NEW}    
		\end{align}
		and, for all $j \in [k]$,
		\begin{align}
			\sum_{i\in[k]}e(G^*_{ij})=\sum_{i\in[k],\,ij\in E(H)}\bigg(e_{G^*}(V_i',w_{ij})+e_{G^*}(w_{ij},V_i')\bigg)= e_{G^*}(\overline{V_j},W_j)+e_{G^*}(W_j,\overline{V_j}).
			\label{eqn:G^*[V_i', W_j]2NEW}    
		\end{align}
		We also need the following inequality
		\begin{align}
			e_G( V_i',V_i) = e_G (V_i')+\sum_{w \in W_i } d^-_G(w, V'_{i} ) \le e_G (V_i')+\sum_{w \in W_i } d^-_G(w, V_{i} ). \label{eqn:k_0eq1}
		\end{align}
		
		Now, let $	k_0 =  |\{i \in [k]: |V_i| \le q d+1\}|$.
		Note that $k_0>0$ since $n<(qd+1)k$ by~\ref{item:relation-n-d-k}.
		Without loss of generality, 
		\begin{align}
			|V_i| \le q d+1 \text{ if and only if } i \in [k_0]. \label{eqn:[k_0}
		\end{align}
		For $i \in [k_0]$, 
		\begin{align}
			d |V'_i| -  e_{G}( V_i' )  & \ge d |V'_i| -  \frac{|V'_i|(|V'_i|-1)}{q}  
			= \frac{|V_i'|}{q} \left( q d +1 - |V'_i| \right) 
			= \frac{|V_i'|}{q} \left( q d +1 - |V_i|  \right)  + \frac{|V_i'||W_i|}{q} \nonumber \\
			&	 \overset{\mathclap{\text{\eqref{eqn:|W_i|}}}}{\ge } \frac{qd-9k}{q} \left( q d +1 - |V_i|  \right) + \frac{|V_i'||W_i|}{q}  \nonumber \\
			& = d \left( q d +1 - |V_i|  \right)  - \frac{9k \left( q d +1 - |V_i|  \right)}{q} + \frac{|V_i'||W_i|}{q} \nonumber \\
			& \overset{\mathclap{\text{\eqref{eqn:sizeV_i}}}}{\ge } d \left( q d +1 - |V_i|  \right)  - 72k^2 + \frac{|V_i'||W_i|}{q}  . \label{eqn:k_0eq2}
		\end{align}
		Then,
		\begin{align}
			e_{G} ( V_i', \overline{V_i} ) 
			&= \sum_{v \in V_i' } d^+_G(v, \overline{V_i} ) 
			= \sum_{v \in V_i' } \left( d - d^+_G(v,  V_{i} )  \right)
			= d |V'_i| -  e_{G}(V_i',V_i) \nonumber \\
			& \overset{\mathclap{\text{\eqref{eqn:k_0eq1}}}}{\ge } d |V'_i| -  e_{G}( V_i' ) - \sum_{w \in W_i } d^-_G(w, V_{i} )	
			= d |V'_i| -  e_{G}( V_i' ) - \sum_{w \in W_i } \left( d - d^-_G(w, \overline{V_{i}} ) \right)
			\nonumber \\
			& = d |V'_i| -  e_{G}( V_i' )  - d |W_i| + e_G ( \overline{V_{i}},W_i  ) \nonumber \\
			& \overset{\mathclap{\text{\eqref{eqn:k_0eq2}}}}{\ge }
			d \left( q d +1 - |V_i|  \right)  - 72k^2 - \frac{|W_i|}{q} (q d - |V_i'|) + e_G ( \overline{V_{i}},W_i  ) \nonumber \\
			& \overset{\mathclap{\text{\eqref{eqn:|W_i|}}}}{\ge } d	\left( q d +1 - |V_i|  \right)  -  81k^2  + e_G ( \overline{V_{i}},W_i  )
			\nonumber
		\end{align}
		and a similar inequality holds for $e_G ( \overline{V_{i}},V_i'  )$. 
		Hence, for $i \in [k_0]$,
		\begin{align}
			e_{G} ( V_i', \overline{V_i} ) + e_G ( \overline{V_{i}},V_i'  ) \ge  2d	\left( q d +1 - |V_i|  \right)  -  162k^2  + e_G ( \overline{V_{i}},W_i  ) + e_G ( W_i,\overline{V_{i}}) .	\label{eqn:e(G[V'_i,barV_i])}
		\end{align}

		Let $\mathcal{I}$ be the set of all connected components~$I$ in~$H^* \left[ [k_0] \right]$ with~$\sum_{ i \in I}  (qd  +1  -  |V_i| ) >0$.\footnote{Here we identify the connected component with its vertex set.}
		Such $\mathcal{I}$ is non empty as $n < (qd +1) k $ by~\ref{item:relation-n-d-k}. 
		We have that
		\begin{align}
			\sum_{i \in [k_0]}  (qd  +1  -  |V_i| ) = \sum_{I\in\mathcal{I}} \sum_{ i \in I}  (qd  +1  -  |V_i| ) .\label{eqn:sum_k_0}
		\end{align}
		Recall that $H^*$ is acyclic and $H$ is an orientation of~$H^*$.
		Note that for each $I\in\mathcal{I}$, $H[I]$ is an oriented tree and $E_H ( I, [k]\setminus [k_0] ) = E_H ( I,[k] ) \setminus E_H (I)$.
		Therefore, for each $I\in\mathcal{I}$, 
		\begin{align*}
			&      (  2-1/k ) d \cdot e_H ( I, [k]\setminus [k_0] )
			= (2-1/k) d \cdot |E_H ( I,[k] ) \setminus E_H (I)|
			\\	& 
			\overset{\mathclap{\eqref{eqn:e(Gij)+e(Gji)}}}{\ge} \sum_{ ij\in E(H) \colon i\in I,\, j\in[k]\setminus I}  e(G_{ij} \cup G_{ji}) \overset{\mathclap{\eqref{eqn:H_ij}}}{\ge}
			\sum_{ij\in E(H) \colon i\in I,\, j\in[k]\setminus I} e(G^*_{ij})
			\\ &= \sum_{ij\in E(H)\colon i\in I,\, j\in[k]} e(G^*_{ij}) - \sum_{ij\in E(H) \colon i,j\in I,\, i\neq j} e(G^*_{ij})
			\\ &\overset{\mathclap{\eqref{eqn:H_ij}}}{\ge} \sum_{ij\in I\times[k]} e(G^*_{ij}) - \sum_{ ij\in [k]\times I} e(G^*_{ij}) \\ &
			\overset{\mathclap{\eqref{eqn:G^*[V_i', W_j]NEW}, \eqref{eqn:G^*[V_i', W_j]2NEW}}}{=}\quad\,\,
			\sum_{i\in I} \bigg(e_{G^*}(V_i', \overline{V_i})+ e_{G^*}( \overline{V_i},V_i') \bigg)- \sum_{ j\in I } \bigg(e_{G^*}(\overline{V_j}, W_j)+e_{G^*}(W_j,\overline{V_j})\bigg) \\ &
			\overset{\mathclap{\eqref{eqn:e(G*)}}}{\ge}
			\,\,\sum_{i\in I} \bigg(e_{G}(V_i', \overline{V_i})+ e_{G}( \overline{V_i},V_i') \bigg)- \sum_{ j\in I } \bigg(e_{G}(\overline{V_j}, W_j)+e_{G}(W_j,\overline{V_j})\bigg) - 3k^4 \\ &
			= \sum_{i\in I} \Bigg(\bigg(e_{G}(V_i', \overline{V_i})+ e_{G}( \overline{V_i},V_i') \bigg)- \bigg(e_{G}(\overline{V_i}, W_i)+e_{G}(W_i,\overline{V_i})\bigg) \Bigg)- 3k^4 \\ &
			\overset{\mathclap{\eqref{eqn:e(G[V'_i,barV_i])}}}{\ge} \,\,\, 2d \sum_{ i \in I}  (qd  +1  -  |V_i| )  - 162|I| k^2 - 3k^4 
			\ge 2d \sum_{ i \in I}  (qd  +1  -  |V_i| )
			- 165 k^4.
		\end{align*}
		After rearranging, for each $I\in \mathcal{I}$, we have 
		\begin{align*}
			e_H(I,[k]\setminus [k_0] )  &\ge  \frac{2}{2-1/k} \sum_{ i \in I}  (qd  +1  -  |V_i| ) 
			- \frac{165 k^4}{(2-1/k)d}  \\
			&= \sum_{ i \in I}(qd  +1  -  |V_i| )+\frac{1}{2k-1}\sum_{ i \in I}(qd  +1  -  |V_i| ) - \frac{165 k^5}{(2k-1)d}\\
			&\ge \sum_{ i \in I}  (qd  +1  -  |V_i| )+\frac{1}{2k-1}-\frac{165 k^5}{(2k-1)d}>\sum_{ i \in I}  (qd  +1  -  |V_i| )
		\end{align*}
		as $\sum_{ i \in I}  (qd  +1  -  |V_i| )\ge 1$ and $d>165k^5$.
		Hence, we have  
		\begin{align}
			\sum_{I\in\mathcal{I}} e_H(I,[k]\setminus [k_0] )  &\ge 
			\sum_{I\in\mathcal{I}} \left( 1 + \sum_{ i \in I}  (qd  +1  -  |V_i| ) \right) 
			= |\mathcal{I}|+ \sum_{I\in\mathcal{I}} \sum_{ i \in I}  (qd  +1  -  |V_i| ) \nonumber \\
			& \overset{\mathclap{\text{\eqref{eqn:sum_k_0}}}}{=}
			|\mathcal{I}|+ \sum_{i \in [k_0]}  (qd  +1  -  |V_i| )
			.
			\label{eqn:k-k_0}
		\end{align}
		Recall that $H$ is an oriented forest and so $H[I]$ is an oriented tree for all $I \in \mathcal{I}$. 
		Then, 
		\begin{align}
			\sum_{I\in \mathcal{I}} e_H (I,[k]\setminus [k_0] ) 
			& \le e_{H} \left( \left([k]\setminus[k_0]\right)\cup \bigcup_{I\in \mathcal{I}} I \right) - \sum_{I\in \mathcal{I}}e_H(I)
			\nonumber
			\\ 
			&\le \left( \left(k-k_0 +\sum_{I\in \mathcal{I}}|I|\right) - 1 \right)-\sum_{I\in \mathcal{I}}\left(|I|-1\right) 
			= k-k_0+|\mathcal{I}|-1. \nonumber
		\end{align}
		Together with \eqref{eqn:k-k_0}, we have 
		\begin{align}
			\sum_{ i \in [k_0]}  (qd  +1  -  |V_i| ) \le k-k_0-1. \label{eqn:sum-of-components-vertices-k0}
		\end{align}
		Finally 
		\begin{align*}
			n & = \sum_{i \in [k]} |V_i| 
			\overset{\mathclap{\eqref{eqn:[k_0}}}{\ge} \sum_{i \in [k_0]} |V_i| + (qd+2) (k-k_0)\\
			&  = (qd+1)k +( k-k_0 ) - \sum_{i \in [k_0]} (qd  +1  -  |V_i| )
			\overset{\mathclap{\eqref{eqn:sum-of-components-vertices-k0}}}{ \ge}  (q d+1)k +1,
		\end{align*}	
		a contradiction to property \ref{item:relation-n-d-k} of Lemma~\ref{lma:regdigraph}.
		This completes the proof of the lemma. 
	\end{proof}

	
	\section{Conclusion} \label{sec:conlusion}
	
	\subsection{Path cover for (non-regular) graphs}
	
	Magnant, Wang and Yuan~\cite{PathCover2} gave a stronger version of Conjecture~\ref{conj:PATH-COVER}. Recall that $\pi(G)$ is the minimum number of vertex-disjoint paths needed to cover $G$.
	
	\begin{conjecture}[Magnant, Wang and Yuan~\cite{PathCover2}]
		\label{conj:PATH-COVER2}
		If $G$ is a graph on $n$ vertices with $\Delta(G) = \Delta$ and $\delta(G) = \delta$, then $\pi(G)\leq \max \{ n/(\delta+1) , (\Delta - \delta) n / (\Delta + \delta) \}$.
	\end{conjecture}
	
	The bound is tight by considering a disjoint union of $K_{\delta+1}$ or a disjoint union of $K_{\delta, \Delta}$.
	The conjecture holds if $\delta \le 2$~\cite{PathCover2} and when $\Delta \ge 2 \delta$~\cite{KouiderZamime}.
	Naturally, one can ask for the directed or oriented analogues.

	\subsection{Edge-disjoint cycles}
	In a weaker version of the problem that we have considered, one is interested in finding \textit{edge-disjoint} cycles whose union covers all the vertices. As a generalization of Dirac's theorem, it was conjectured by Enomoto, Kaneko and Tuza~\cite{ENO-KA-TU} that if a graph $G$ on $n$ vertices has minimum degree at least $n/k$, then $V(G)$ can be covered by $k-1$ edge-disjoint cycles. The case $k=3$ was also proved in~\cite{ENO-KA-TU}. The conjecture was proved for $2$-connected graphs~\cite{Edge-Disjoint-2-connected}, and has been completely resolved in~\cite{Edge-Disjoint-Full}. Later, Balogh, Mousset and Skokan~\cite{Edge-Disjoint-Stability} obtained a stability result, showing that every graph on $n$ vertices with minimum degree nearly $n/k$ has a special structure if it does not have $k-1$ edge-disjoint cycles covering all vertices. One can ask the same question for digraphs as a generalization of Ghoulia-Houri's theorem~\cite{GhouilaHouri}, and Theorem~\ref{thm:main} answers it affirmatively for regular digraphs:
	\begin{conjecture}\label{conj:DIRECTED-EDGE-DISJOINT}
		Let $k\in\mathbb{N}$ with $k\geq 2$. If $G$ is a digraph on $n$ vertices with minimum semi-degree at least~$n/k$, then $V(G)$ can be covered by $k-1$ edge-disjoint cycles.  
	\end{conjecture}

	\subsection{Connectivity and regularity}
	Jackson's conjecture states that imposing regularity on an oriented graph reduces the degree threshold for Hamiltonicity.
	One might hope that imposing connectivity on regular oriented graphs can reduce the degree threshold for Hamiltonicity.
	We refer the reader to \cite[Sections 1 and~7]{LPY} for history and conjectures on Hamiltonicity in regular (directed or oriented) graphs with given connectivity. 

	Similarly, in the setting of cycle partitions, one might hope that connectivity in addition to regularity might reduce the upper bound in Conjecture~\ref{conj:PATH-COVER}. In the sparse setting, Reed~\cite{Reed-Connected} proved that every $3$-regular connected $n$-vertex graph can be covered by at most $\lceil n/9 \rceil$ vertex-disjoint paths, and conjectured that it suffices to use $\lceil n/10 \rceil$ paths if connectivity is replaced by $2$-connectivity. Recall that Conjecture~\ref{conj:PATH-COVER} gives the upper bound of $n/4$ for $3$-regular graphs (that are not necessarily connected). Yu~\cite{Yu-Verifies-Reed} recently verified Reed's conjecture and gave an example of a (2-connected) $d$-regular graph on $n$ vertices which requires at least $\approx n/(d+4)$ paths for $d\ge 13$. 
	It would be interesting to investigate the general relationship between the degree and connectivity of a regular (di)graph or oriented graph that guarantees a small number of vertex-disjoint cycles that cover all the vertices.

\end{document}